\documentclass[11pt]{article}

\usepackage[margin=1in]{geometry}  
\usepackage{amsmath,amssymb,amsthm,amsfonts}  
\usepackage{bm}  
\usepackage{graphicx}
\usepackage{caption}
\usepackage{subcaption}
\usepackage{microtype}
\usepackage{enumerate}
\usepackage[shortlabels]{enumitem}
\usepackage[numbers]{natbib}
\usepackage{doi}
\usepackage{hyperref}
\usepackage{cleveref}
\usepackage{placeins}
\usepackage{comment}
\usepackage{algorithm}
\usepackage[noend]{algpseudocode}
\usepackage{float}
\usepackage[title]{appendix}
\usepackage[nottoc]{tocbibind}
\usepackage{booktabs}

\hypersetup{
colorlinks=true,  
linkcolor=blue,  
citecolor=blue,  
urlcolor=blue,  
pdftitle={Coordinated Mean-Field Control for Systemic Risk},  
pdfauthor={Toshiaki Yamanaka}}

\theoremstyle{plain}
\newtheorem{theorem}{Theorem}[section]
\newtheorem{lemma}[theorem]{Lemma}
\newtheorem{corollary}[theorem]{Corollary}
\newtheorem{proposition}[theorem]{Proposition}

\theoremstyle{definition}
\newtheorem{definition}[theorem]{Definition}
\newtheorem{assumption}[theorem]{Assumption}

\theoremstyle{remark}
\newtheorem{remark}[theorem]{Remark}

\crefname{equation}{Eq.}{Eqs.}

\title{Coordinated Mean-Field Control for Systemic Risk}
\author{Toshiaki Yamanaka \thanks{Whiting School of Engineering, Johns Hopkins University, Baltimore, MD, USA. Email: \href{mailto:tyamana1@jhu.edu}{tyamana1@jhu.edu}}}
\date{December 4, 2025}

\begin{document}  
\maketitle

\begin{abstract}
We develop a robust linear-quadratic mean-field control framework for systemic risk under model uncertainty, in which a central bank jointly optimizes interest rate policy and supervisory monitoring intensity against adversarial distortions. Our model features multiple policy instruments with interactive dynamics, implemented via a variance weight that depends on the policy rate, generating coupling effects absent in single-instrument models. We establish viscosity solutions for the associated HJB--Isaacs equation, prove uniqueness via comparison principles, and provide verification theorems. The linear-quadratic structure yields explicit feedback controls derived from a coupled Riccati system, preserving analytical tractability despite adversarial uncertainty. Simulations reveal distinct loss-of-control regimes driven by robustness-breakdown and control saturation, alongside a pronounced asymmetry in sensitivity between the mean and variance channels. These findings demonstrate the importance of instrument complementarity in systemic risk modeling and control.
\end{abstract}

\medskip
\noindent\textbf{Keywords}: Mean-field control, linear-quadratic mean-field control, systemic risk, central banking

\medskip
\noindent\textbf{Mathematics Subject Classification}: 49N10, 49N80, 93E20

\tableofcontents

\section{Introduction} \label{sect:intro}
\paragraph{Motivation.}
In this paper, we study systemic risk by integrating insights from robust control and mean-field theory. Modern CBs (central banks) frequently operate as both \emph{liquidity providers} and \emph{prudential supervisors}. However, as detailed below, prior studies in financial mathematics do not incorporate multiple policy measures within a single robust control framework, and our work proposes a unified framework to address this limitation.

While the full model and dynamics are presented in \cref{sect:model}, we begin by briefly introducing our model of multiple CB policy measures. First, CBs control interest rates via a policy rate $u_t$, a common component of banks' funding costs. When $u_t$ rises, maintaining liquidity becomes more expensive system-wide. We capture this by letting the variance weight depend on the policy rate, replacing the constant $w_2$ with $w_2(u_t)$ in the term $w_2(u_t)v_t$, where $w_2(u_t)=\bar w_2+\kappa u_t$, and $\bar w_2 + \kappa u_{\min}>0$ ensures uniform positivity. $\bar w_2$ is the baseline penalty weight on variance.

Second, CBs allocate supervisory resources via a monitoring intensity $\pi_t$. Greater scrutiny mitigates systemic dispersion but entails administrative costs captured by $R \pi_t^2$. Its effect on variance enters the variance dynamics via $-\chi \pi_t$ in $\dot v_t$. Within our robust LQ (linear-quadratic)-MFC (mean-field control) framework with adversarial distortions bounded by relative-entropy (Kullback--Leibler [KL] \cite{KullbackLeibler1951}) divergence, the CB jointly optimizes $u_t$ and $\pi_t$ to minimize a quadratic objective in the cross-sectional moments $(m_t,v_t)$, subject to admissible bounds $u_t\in[u_{\min},u_{\max}]$ and $\pi_t\in[0,\pi_{\max}]$.

This structure aligns with the institutional practice at the Federal Reserve, the ECB (European Central Bank), and the BoE (Bank of England), while preserving analytical tractability through linear dynamics, quadratic costs, and closed-form linear feedback. Both monetary and prudential policies are grounded in law,\footnote{Federal Reserve Act, Sections 2A and 21(4); Statute of the European System of Central Banks and of the European Central Bank, Articles 2, 18, and 25.2; Bank of England Act 1998, Sections 11 and 2A(1).} and our model is consistent with this institutional foundation.

\paragraph{Related literature.}
Financial contagion and systemic risk have been extensively studied in financial mathematics.\footnote{Closely related studies in economics include Freixas, Parigi, and Rochet~\cite{Freixas2000}, who analyze systemic risk in an interbank market and CB liquidity provision. Gai and Kapadia~\cite{Gai2010} develop a model of contagion in financial networks and identify phase transitions.} One line of research focuses on the \emph{network structure} of the financial system. Fouque and Ichiba~\cite{FouqueIchiba2013} proposed a diffusion model of interbank lending that captures how banks' lending preferences can lead to multiple defaults. Cont, Moussa, and Santos~\cite{ContMoussaSantos2013} introduced a metric for the systemic importance of financial institutions—the Contagion Index—to quantify contagion and systemic risk in a network of financial institutions. Amini, Cont, and Minca~\cite{AminiContMinca2016} analyzed distress propagation in large financial networks and established rigorous asymptotic results for the magnitude of contagion. Amini, Filipovi{\'c}, and Minca~\cite{AminiFilipovicMinca2020} studied how clearing all contracts through a central node affects a financial network.

Another major line of research adopts \emph{mean-field} models. The general theory of mean-field systems was pioneered by Lasry and Lions~\cite{LasryLions2007} and by Huang, Malham\'e, and Caines~\cite{Huang2006}. Comprehensive and foundational treatments are given in Bensoussan, Frehse, and Yam~\cite{BensoussanFrehseYam2013}, Carmona~\cite{Carmona2016}, and Carmona and Delarue~\cite{CarmonaDelarue2018I, CarmonaDelarue2018II}.\footnote{Furthermore, linear-quadratic-Gaussian games with one major player interacting with a large number of minor players were analyzed by Huang~\cite{Huang2010}. Mean-field games between a dominating player and representative agents were studied by Bensoussan, Chau, and Yam~\cite{Bensoussan2016}.} Within the context of financial systemic risk, various studies have applied the mean-field framework. Carmona, Fouque, and Sun~\cite{CarmonaFouqueSun2015} proposed an MFG (mean-field game) model of interbank lending and borrowing, formulating the evolution of banks' log-monetary reserves as a system of diffusion processes coupled through their drifts. Bo and Capponi~\cite{BoCapponi2015} developed a mean-field model where banks are subject to sudden shocks affecting their monetary reserves. Sun~\cite{Sun2018} proposed an MFG model with an LQ cost structure, in which the CB acts as a central deposit institution. Hambly and S{\o}jmark~\cite{HamblySojmark2019} introduced a dynamic mean-field model for systemic risk in large financial systems, where the mean-field limit is characterized by a nonlinear SPDE (stochastic PDE). Feinstein and S{\o}jmark~\cite{FeinsteinSojmark2021} proposed a dynamic contagion model with endogenous early defaults for a finite set of banks, reformulated as a stochastic particle system leading to a mean-field problem. Cuchiero, Reisinger, and Rigger~\cite{CuchieroReisingerRigger2024} studied an MFC problem and computed the CB's optimal strategy via a PG (policy gradient) method, where the CB controls the rate of capital injections to distressed institutions in order to limit defaults. Bayraktar, Guo, Tang, and Zhang~\cite{Bayraktaretal2025} studied the problem of capital provision arising from systemic risk in a financial network modeled by SDEs, adopting a mean-field particle system approach.

Furthermore, Minca and Sulem~\cite{MincaSulem2014} formulated an optimization problem for a government with a constrained budget seeking to maximize the total net worth of a financial system of banks and their creditors. Cont, Guo, and Xu~\cite{ContGuoXu2021} analyzed stochastic differential games involving singular controls, motivated by a dynamic model of interbank lending with benchmark rates. Veraart and Aldasoro~\cite{VeraartAldasoro2025} developed a framework for modeling risk and quantifying payment shortfalls in cleared markets with multiple central counterparties.

Comprehensive expositions of stochastic/optimal controls include Yong and Zhou~\cite{YongZhou1999}, Fleming and Soner~\cite{FlemingSoner2006}, Hansen and Sargent~\cite{HansenSargent2008}, Pham~\cite{Pham2009}, and Bensoussan~\cite{Bensoussan2018}. The LQ-MFC problem was considered by Carmona and Delarue~(\cite{CarmonaDelarue2018I}, Subsection 6.7.1). The LQ-MFC framework has been widely studied due to its analytical tractability and broad range of applications (\emph{e.g.}, Basei and Pham~\cite{BaseiPham2019} and Yong~\cite{Yong2013}). In the context of systemic risk, however, applications of the LQ-MFC framework to monetary policy transmission in banking remain largely unexplored. In a recent preprint, De Crescenzo, De Feo, and Pham~\cite{DeCrescenzoDeFeoPham2025} introduced an LQ non-exchangeable MFC problem that generalizes the LQ-MFC framework by incorporating heterogeneous interactions.

\paragraph{Scope and research positioning.}
We focus on the aggregate liquidity management aspect of systemic risk, where a CB coordinates system-wide liquidity through multiple policy instruments. Our framework captures how common shocks and cross-sectional dispersion in liquidity create systemic vulnerabilities that require coordinated policy responses. The mean-reversion term $-\beta(L_t^i-m_t)$ in our model represents interbank adjustment mechanisms, and the variance dynamics capture heterogeneous stress across the banking sector.

\paragraph{Our contributions.}
Our paper makes three primary contributions to the literature.
\begin{enumerate}
\item We integrate robust control against model uncertainty into the LQ-MFC framework for systemic risk, allowing an adversary to distort the drift ($\theta$) and variance dynamics ($\xi$) to capture worst-case model misspecification. We establish viscosity solutions for the resulting HJBI (HJB--Isaacs) equation and prove verification theorems that connect PDE solutions to optimal strategies.
\item Unlike prior studies in financial mathematics, we model the joint optimization of interest rate policy and supervisory monitoring intensity with interactive dynamics via state-dependent variance weight $w_2(u_t) = \bar{w}_2 + \kappa u_t$. This coupling, absent in single-instrument models, captures heterogeneous transmission mechanisms and institutional realities.
\item Our analysis reveals complementarity between interest rate and monitoring policies under model uncertainty. Simulations show phase transitions from controllable to uncontrollable regimes, with asymmetric burdens on monetary versus supervisory tools—a phenomenon critical under model uncertainty.
\end{enumerate}

The LQ structure preserves analytical tractability, yielding explicit Riccati equations and closed-form feedback policies that remain computationally feasible even with control bounds and state constraints. Owing to this analytical tractability, our model provides a tractable baseline that admits various meaningful extensions, as discussed in \cref{sect:discuss}.

\paragraph{Comparison to prior work.} While Sun~\cite{Sun2018} and Cuchiero, Reisinger, and Rigger~\cite{CuchieroReisingerRigger2024} are related to our setting, our framework uniquely captures how a CB's policy rate influences liquidity dispersion through the coupling parameter $\kappa > 0$, providing a direct channel from monetary policy to systemic stability. Unlike De Crescenzo, De Feo, and Pham~\cite{DeCrescenzoDeFeoPham2025}, our approach captures heterogeneous policy transmission through $\kappa > 0$ within an exchangeable framework.

\paragraph{Outline of the paper.}
The remainder of this paper is organized as follows. \Cref{sect:model} introduces the LQ-MFC framework with multiple policy instruments and robust control against adversarial distortions. \Cref{sect:theor} establishes the theoretical foundations, including viscosity characterization, verification theorems, and the quadratic ansatz with its associated Riccati system. \Cref{sect:simul} presents a comprehensive numerical analysis examining adversary strength, parameter sensitivity, and loss-of-control regimes. \Cref{sect:discuss} discusses limitations and extensions, and \cref{sect:concl} concludes. Technical proofs and propagation-of-chaos analysis are provided in the appendices.

\section{Model and dynamics} \label{sect:model}
\paragraph{Model setting.}
Let $L_t^{i}$ denote the liquidity gap of bank $i \in [0,1]$, a continuum of agents. The CB chooses a policy rate $u_t$ and a monitoring intensity $\pi_t$ from admissible sets $u_t \in \mathcal U := [u_{\min}, u_{\max}]$ and $\pi_t \in \mathcal P := [0, \pi_{\max}]$. Controls $(u,\pi)$ and adversarial distortions $(\theta,\xi)$ are progressively measurable and square-integrable. Denote the cross-sectional mean and variance by $m_t := \mathbb{E}[L_t^{i}]$ and $v_t := \mathrm{Var}[L_t^{i}]$. Individual dynamics follow a linear McKean--Vlasov SDE with common and idiosyncratic Brownian motions $(B_t, W_t^{i})$:
\[dL_t^{i} = \bigl[-\beta\,(L_t^{i}-m_t) + \eta\,u_t + \theta_t\bigr]dt
 \, + \, \sigma_L\,dW_t^{i} \, + \, \sigma_c\,dB_t,\]
where $\beta>0$ captures interbank netting effects and $\eta>0$ is the pass-through from $u_t$ to funding costs. The term $\theta_t$ is a worst-case drift distortion chosen by an adversary.

\begin{remark}[mean reversion mechanism]
The parameter $\beta > 0$ in the drift term $-\beta(L_t^i - m_t)$ captures the mean-reverting nature of liquidity dynamics via interbank netting and clearing mechanisms. This term generates a process in which deviations from the mean $m_t$ decay at rate $\beta$. When $L_t^i > m_t$ (excess liquidity), the negative drift pulls the bank's position downward, while $L_t^i < m_t$ (liquidity shortage) induces an upward drift. A higher $\beta$ represents more efficient interbank markets with faster redistribution of liquidity imbalances, while a lower $\beta$ reflects frictions in interbank interactions.
\end{remark}

To derive the aggregate dynamics, we take expectations of $dL_t^i$, which yields $\dot{m}_t$ because the linear mean-reversion terms cancel at the aggregate level. The variance dynamics follow from the model specification:
\[\dot m_t \;=\; \eta\,u_t + \theta_t, \quad \dot v_t \;=\; -2\beta\,v_t + \sigma_L^2 + \sigma_c^2 + \xi_t - \chi\,\pi_t,\]
where $\sigma_L^2 + \sigma_c^2$ is the effective variance forcing (Remark~\ref{rema:variance}), $\xi_t$ is a worst-case dispersion distortion, and $\chi>0$ measures the effectiveness of monitoring on variance.

\begin{remark}[model specification]
We adopt a specification $\dot m_t \;=\; \eta\,u_t + \theta_t$ for the mean dynamics. This modeling choice preserves the LQ structure necessary for deriving closed-form solutions. It allows us to separate and identify two distinct channels of monetary policy transmission: the direct effect on mean liquidity through $\eta$ and the heterogeneous effect on liquidity dispersion through the variance penalty coupling $w_2(u_t) = \bar{w}_2 + \kappa u_t$. While richer dynamics may be appealing, they may obscure these transmission mechanisms and sacrifice analytical tractability.
\end{remark}

\begin{remark}[variance dynamics and common noise]\label{rema:variance}
While the underlying system includes common noise $\sigma_c dB_t$, standard aggregation implies that this term cancels in the cross-sectional dynamics $d(L_t^i - m_t)$, leaving only idiosyncratic volatility $\sigma_L$ in the variance drift. We retain $\sigma_c^2$ in the variance dynamics $\dot v_t$ as a conservative modeling choice that accounts for potential additional dispersion channels, such as heterogeneous bank sensitivities to common shocks, within the tractable LQ framework.
\end{remark}

Under the parameter restriction $\sigma_L^2 + \sigma_c^2 \geq \chi\pi_{\max}$, the variance remains non-negative. This condition suffices because the optimal adversarial distortion $\xi_t^* = 2\lambda_v \partial_v V \geq 0$ increases variance, so the binding constraint for non-negativity at $v=0$ is maximum monitoring with no adversarial pressure. If $v_t$ reaches zero, the drift $\sigma_L^2 + \sigma_c^2 + \xi_t - \chi\pi_t \geq 0$ ensures it cannot become negative. The common noise $B_t$ implies a conditional McKean--Vlasov limit, where propagation-of-chaos and limit statements are understood conditional on the common filtration (see Appendix~\ref{appe:A}). While we solve the problem at the level of deterministic moment dynamics $(m_t, v_t)$, Appendix~\ref{appe:A} shows that, as $N \to \infty$, the empirical mean and variance of the stochastic $N$-bank system converge to $(m_t, v_t)$.

The CB's objective is to stabilize the system by minimizing the aggregate liquidity gap and its dispersion while limiting control efforts. Specifically, the CB minimizes a quadratic mean-field objective composed of running penalties on the squared mean gap $m_t^2$, the cross-sectional variance $v_t$, and quadratic costs associated with monitoring $\pi_t^2$, and policy rate adjustments $u_t^2$. We include a terminal cost $G_m m_T^2 + G_v v_T$ to ensure the system is steered toward stability by the terminal time $T$, penalizing any remaining aggregate imbalance $m_T$ or dispersion $v_T$. To capture the interaction between monetary policy and variance $v_t$, we introduce a variance weight $w_2(u_t)$ that depends on the policy rate:
\[J(u,\pi) =\int_0^T \Big( w_1m_t^2+w_2(u_t)v_t+R\pi_t^2+R_uu_t^2 \Big)dt+G_mm_T^2+G_vv_T, \quad w_2(u_t)=\bar w_2 + \kappa u_t,\]
with $\bar w_2>0$ and $\kappa>0$. We assume $R_u>0$, $R>0$, and $\bar w_2 + \kappa\,u_{\min} > 0$ so that $w_2(u_t)$ is uniformly positive on $\mathcal U$.

\begin{remark}[terminal variance and cost structure]\label{rema:terminal}
Under the linear terminal penalty $G_v v_T$, the variance $v_T$ settles at a non-zero optimal equilibrium where the marginal cost of further variance reduction equals its marginal benefit against adversarial pressure $\xi$ and system noise $\sigma_L^2+\sigma_c^2$. We adopt a linear penalty on $v_T$ because variance already represents the second moment of liquidity gaps ($v_t = \mathbb{E}[(L_t^i - m_t)^2]$, the cross-sectional variance). Thus, a linear penalty on $v_t$ constitutes a quadratic penalty on the underlying bank positions, preserving the LQ structure. A quadratic penalty $G_v v_T^2$ would penalize fourth moments, yielding vanishing marginal incentives near zero ($\partial_v(v^2) = 2v \to 0$). Our numerical experiments confirm that this quadratic penalty leads to higher terminal variance $v_T$ due to these weakened control incentives.
\end{remark}

\begin{remark}[motivation for $\kappa > 0$]
The coupling term $w_2(u_t) = \bar w_2 + \kappa u_t$ captures the state-dependent nature of variance penalties, reflecting that monetary tightening affects banks heterogeneously based on their liquidity positions. In a general control setting, the cost of cross-sectional dispersion may depend on the aggregate policy stance, represented by a general weight function $\mathcal{W}(u_t)$. We assume $\mathcal{W}$ is smooth and admits a Taylor expansion around the neutral rate $u^*=0$: $\mathcal{W}(u_t) \approx \mathcal{W}(0) + \mathcal{W}'(0) u_t + \mathcal{O}(u_t^2).$ Identifying $\bar{w}_2 = \mathcal{W}(0)$ and $\kappa = \mathcal{W}'(0)$, our specification represents the first-order truncation of this general dependency. The assumption $\kappa > 0$ implies that the marginal cost of dispersion increases with the policy rate (a tightening regime amplifies the penalty on heterogeneity). Retaining only the linear term preserves the LQ structure of the problem, allowing for the explicit Riccati characterization derived in \cref{sect:quadratic}. Including quadratic or higher-order terms in the weight would render the Hamiltonian non-quadratic in $u$, destroying the Riccati property and requiring numerical methods that obscure the analytical characterization of optimal policy.
\end{remark}

\paragraph{Robustness.}
Robustness is imposed via a relative-entropy budget on the adversary. Nature selects $(\theta_t,\xi_t)$ subject to:
\begin{equation}\label{eq:robust}
\int_0^T \left( \frac{\theta_t^2}{\lambda_m} + \frac{\xi_t^2}{\lambda_v} \right) dt \;\le\; \rho,
\end{equation}
with $\lambda_m,\lambda_v>0$ and budget $\rho>0$. By convex duality, this constrained formulation is equivalent to a penalized (Lagrangian) formulation in which the running cost includes terms $-\frac{\theta_t^2}{4\lambda_m} - \frac{\xi_t^2}{4\lambda_v}$, with $\lambda_m$ and $\lambda_v$ acting as Lagrange multipliers. We adopt this penalized formulation henceforth. The HJBI equation includes convex penalties that bound worst-case distortions and ensure well-posed Riccati equations. These relative-entropy penalties yield bounded linear worst-case feedback and ensure the Isaacs condition.

\paragraph{Value function and Riccati system.}
The coupling term $\kappa u_t v_t$ in the cost functional (through $w_2(u_t) = \bar{w}_2 + \kappa u_t$) necessitates a full quadratic ansatz in both $m$ and $v$, including cross-terms. With terminal cost $G_m m_T^2 + G_v v_T$, we seek a value function of the form given in \cref{eq:ansatz}, and the HJBI reduces to a coupled system of Riccati ODEs in \cref{eq:Riccati}.

\paragraph{Feedback laws.}
We write $\partial_m V$ and $\partial_v V$ for the partial derivatives of $V$ with respect to $m$ and $v$, respectively. Solving the HJBI (derived via the Isaacs Hamiltonian in \cref{sect:HJBI}) yields linear feedback laws, implemented with projection onto admissible sets. Specifically, minimizing the Isaacs Hamiltonian $H(x, \nabla V)$ with respect to the controls $u$ and $\pi$ (before projection) yields the first-order conditions $2R_u u + (\eta \partial_m V + \kappa v) = 0$ and $2R \pi - \chi \partial_v V = 0.$ Solving these yields the linear feedback laws:
\[u_t^{\mathrm{fb}}=-\dfrac{\kappa v_t+\eta \partial_m V}{2R_u}, \qquad \pi_t^{\mathrm{fb}} \;=\; \frac{\chi\,\partial_v V}{2R}.\]
The projected controls are:
\[u_t^{*} \;=\; \Pi_{[u_{\min},u_{\max}]}\!\bigl(u_t^{\mathrm{fb}}\bigr), \qquad \pi_t^{*} \;=\; \Pi_{[0,\pi_{\max}]}\!\bigl(\pi_t^{\mathrm{fb}}\bigr).\]
Worst-case distortions are linear in value gradients and bounded:
\[\theta_t^{*} \;=\; 2\,\lambda_m\,\partial_m V, \qquad \xi_t^{*} \;=\; 2\,\lambda_v\,\partial_v V.\]

Comparative statics are piecewise-smooth with potential kinks at projection boundaries. For simulation, we restrict parameters to the region where the Riccati system admits a bounded solution.

\begin{remark}
We adopt the penalized robust-control formulation, which is dual to a relative-entropy budget and implies that the Isaacs condition holds under our convexity/concavity assumptions. Parameters satisfy $R>0$, $R_u>0$, and $\min_{u\in[u_{\min},u_{\max}]}(\bar w_2+\kappa u)>0$, so the variance weight is uniformly positive. Interior feedback is implemented via projection onto admissible sets, and comparative statics are evaluated away from switching times.
\end{remark}

\section{Theoretical foundations}\label{sect:theor}
This section develops a framework linking the control formulation to solution concepts and implementable policies for robust mean-field models. We establish the DPP (dynamic programming principle, see, \emph{e.g.,} Fleming and Soner~\cite{FlemingSoner2006}) and its associated HJBI equation. We then prove existence and uniqueness of viscosity solutions under regularity assumptions. A comparison principle yields uniqueness, and a verification theorem translates PDE solutions into optimal strategies, together providing a well-posed and operational framework for robust control.

Building on the general theory, we specialize to a robust LQ-MFC setting. Via square completion, we obtain explicit Isaacs and saddle-point structures and a quadratic value function, verified by a Riccati system. We derive closed-form feedback policies and conclude with comparative statics and robustness loss bounds that quantify sensitivity to parameters and misspecification (the details are provided in Appendix~\ref{appe:C}). Collectively, these results deliver both the theoretical guarantees and practical tools needed to analyze and implement robust policies in large-scale mean-field environments.

\subsection{Viscosity solutions and the HJBI equation}\label{sect:viscosity_and_HJBI}
In this subsection, we establish the viscosity characterization of the robust HJBI equation. 

\subsubsection{Model primitives and admissible inputs}\label{sec:model_primitives}
As introduced in \cref{sect:model}, let $X_t=(m_t,v_t)\in \mathbb{R}\times \mathbb{R}_+$ denote the moment state, with absolutely continuous dynamics $\dot X_t = b(X_t, u_t, \pi_t, \theta_t, \xi_t)$, where the drift function $b: (\mathbb{R} \times \mathbb{R}_+) \times \mathcal{U} \times \mathcal{P} \times \mathbb{R}^2 \to \mathbb{R}^2$ is defined by $b(x,u,\pi,\theta,\xi) := \big(\eta u + \theta,\; -2\beta v + \sigma_L^2 + \sigma_c^2 + \xi - \chi\pi\big),$ where $x = (m,v)$, $u\in \mathcal U=[u_{\min},u_{\max}]$, and $\pi\in \mathcal P=[0,\pi_{\max}]$ are progressively measurable controls, while $(\theta,\xi)$ are progressively measurable distortions.

The penalized running cost function is:
\[ \ell(m,v,u,\pi,\theta,\xi) := w_1 m^2 + w_2(u)\,v + R\,\pi^2 + R_u\,u^2 - \tfrac{\theta^2}{4\lambda_m} - \tfrac{\xi^2}{4\lambda_v},\]
with terminal cost function $g(x):= G_m m^2 + G_v v$, where $x = (m,v)$ and $w_2(u)=\bar w_2 + \kappa u$.

\begin{assumption}[standing assumptions] \label{assu:A}
We assume the following.

\begin{enumerate}
\item $\mathcal U,\mathcal P$ are compact intervals. The processes $u_t, \pi_t, \theta_t,$ and $\xi_t$ are progressively measurable and square-integrable.
\item For any admissible inputs $(u_t, \pi_t, \theta_t, \xi_t)$, the controlled ODE $\dot X_s = b\left(X_s, u_s, \pi_s, \theta_s, \xi_s\right)$ with initial condition $X_t = x$ admits a unique absolutely continuous solution on $[t, T]$. Moreover, the solution has at most linear growth in the state.
\item The functions $\ell$ and $g$ are continuous in all their arguments. The running cost function $\ell$ is convex in $(u,\pi)$ and concave in $(\theta,\xi)$. The map $w_2$ is continuous in $u$.
\item The admissible inputs $(u_t, \pi_t, \theta_t, \xi_t)$ are closed under concatenation at stopping times, and the cost functional is additive over time, ensuring the DPP.
\end{enumerate}
\end{assumption}

Based on the framework of Petersen, James, and Dupuis~\cite{Petersen2000}, we adopt the following formulation.

\begin{definition}[lower and upper values]
For $(t,x) \in [0,T] \times (\mathbb{R} \times \mathbb{R}_+)$, define the lower value
\[V(t,x) := \inf_{(u,\pi)} \sup_{(\theta,\xi)}\!\left[\int_t^T \ell\!\left(X_s, u_s, \pi_s, \theta_s, \xi_s\right)\, ds \;+\; g\!\left(X_T\right)\right],\]
and the upper value
\[\widehat V(t,x) := \sup_{(\theta,\xi)} \inf_{(u,\pi)}\!\left[\int_t^T \ell\!\left(X_s, u_s, \pi_s, \theta_s, \xi_s\right)\, ds \;+\; g\!\left(X_T\right)\right],\]
where the infimum and supremum are taken over admissible inputs.
\end{definition}

\begin{proposition}[DPP and terminal condition]\label{prop:DPP}
Under Assumption~\ref{assu:A}, both the lower value $V$ and the upper value $\widehat V$ satisfy the DPP, and $V(T,x) = \widehat V(T,x) = g(x), \quad x \in \mathbb{R} \times \mathbb{R}_+.$
\end{proposition}

\subsubsection{HJBI and viscosity characterization}\label{sect:HJBI}
\begin{definition}[Isaacs Hamiltonian]
Let $x=(m,v)$ denote the state and $p=(p_m,p_v)\in\mathbb{R}^2$ denote the adjoint variables. Define
\begin{equation}\label{eq:IsaacsH_1}
H(x,p) := \inf_{u\in \mathcal U,\ \pi\in \mathcal P}\ \sup_{\theta,\xi\in \mathbb{R}} \Big\{ p_m(\eta u+\theta) + p_v\big(-2\beta v + \sigma_L^2+\sigma_c^2+\xi-\chi\pi\big) + \ell(m,v,u,\pi,\theta,\xi) \Big\}.
\end{equation}
Direct maximization in $(\theta,\xi)$ yields
\[\sup_{\theta,\xi\in\mathbb{R}} \Big\{ p_m\theta + p_v\xi - \tfrac{\theta^2}{4\lambda_m} - \tfrac{\xi^2}{4\lambda_v} \Big\} \;=\; \lambda_m p_m^2 + \lambda_v p_v^2,\]
such that
\begin{equation}\label{eq:IsaacsH}
\begin{aligned}
H(x,p) = &\lambda_m p_m^2 + \lambda_v p_v^2 \\ & + \inf_{u\in\mathcal{U},\pi\in\mathcal{P}} \Big\{ w_1 m^2 + w_2(u) v + R\pi^2 + R_u u^2 + p_m \eta u + p_v(-2\beta v + \sigma_L^2+\sigma_c^2-\chi\pi)\Big\}.
\end{aligned}
\end{equation}
\end{definition}

\begin{proposition}[HJBI for the value function]\label{prop:HJBI}
Under Assumption~\ref{assu:A} and the Isaacs condition (see Lions~\cite{Lions1983} and Fleming and Soner~\cite{FlemingSoner2006}, Eq. XI(3.11)), the lower value $V$ is a viscosity solution of the HJBI equation
\begin{equation}\label{eq:HJBI}
-\partial_t V(t,x) + H\!\big(x,\nabla_x V(t,x)\big) = 0, \quad V(T,x)=g(x).
\end{equation}
Moreover, the upper value $\widehat V$ is also a viscosity solution of
\[-\partial_t \widehat V(t,x) + H\!\big(x,\nabla_x \widehat V(t,x)\big) = 0, \quad \widehat V(T,x)=g(x).\]
\end{proposition}

\begin{theorem}[viscosity characterization of the robust HJBI]\label{theo:visc}
Under Assumption~\ref{assu:A} and the DPP for $V$ and $\widehat V$, the following hold:
\begin{enumerate}
\item $V$ is a viscosity supersolution of $-\partial_t \phi + H(x,\nabla_x \phi)=0$ on $[0,T)\times(\mathbb{R}\times\mathbb{R}_+)$, bounded from below with at most polynomial growth, and satisfies $V(T,\cdot)=g(\cdot)$.
\item $\widehat V$ is a viscosity subsolution of the same equation with $\widehat V(T,\cdot)=g(\cdot)$.
\item If Isaacs' condition holds (by convexity in $(u,\pi)$, concavity in $(\theta,\xi)$, and compactness), then $V=\widehat V$ and the common value is a viscosity solution of the HJBI.
\end{enumerate}
\end{theorem}

The proof follows viscosity arguments for robust control problems and is deferred to Appendix~\ref{proo:visc}.

\subsubsection{Comparison principle for the robust HJBI}\label{sect:comparison}
Let $x=(m,v)\in \mathbb{R}\times\mathbb{R}_+$ and define the Isaacs Hamiltonian as in \cref{eq:IsaacsH}. Consider the HJBI $-\partial_t V(t,x) + H\big(x,\nabla_x V(t,x)\big) = 0, \quad (t,x)\in [0,T)\times(\mathbb{R}\times\mathbb{R}_+),$ with terminal condition $V(T,x)=g(x)=G_mm^2+G_vv$.

\begin{assumption}[structural and growth conditions]\label{assu:B}
We assume the following.
\begin{enumerate}
\item $R_u \ge c_u > 0$ and $\min_{u\in[u_{\min},u_{\max}]} w_2(u) \ge c_w > 0$.
\item The Hamiltonian $H(x,p)$ is continuous in $(x,p)$, locally Lipschitz in $x$ on bounded sets, with at most polynomial growth in $x$ and at most quadratic growth in $p$.
\item Any viscosity subsolution and supersolution considered are continuous, satisfy for some $C,k$ the bound $|U(t,x)|\le C(1+|x|^k)$ uniformly in $t$, and attain the terminal condition in the viscosity sense.
\item State-constraint boundary at $v=0$. We work on the closed set $\mathbb{R}\times\mathbb{R}_+$ with \emph{constrained viscosity solutions} in the sense of Soner~\cite{Soner1986II}. No boundary condition is prescribed at $v=0$. Viscosity inequalities are tested with constrained semijets (\emph{i.e.}, using interior test functions).
\end{enumerate}
\end{assumption}

\begin{theorem}[comparison principle and uniqueness]\label{theo:comparison}
Let $U$ be a bounded-from-below, polynomial growth viscosity subsolution of
\[-\partial_t U + H(x,\nabla_x U) \le 0 \quad \text{on } [0,T)\times(\mathbb{R}\times\mathbb{R}_+),\]
and $W$ be a viscosity supersolution of
\[-\partial_t W + H(x,\nabla_x W) \ge 0 \quad \text{on } [0,T)\times(\mathbb{R}\times\mathbb{R}_+),\]
with $U(T,\cdot)\le g(\cdot)\le W(T,\cdot)$ and the growth in Assumption~\ref{assu:B}.

Then
$U(t,x) \le W(t,x) \text{ for all } (t,x)\in [0,T]\times(\mathbb{R}\times\mathbb{R}_+).$ Consequently, the viscosity solution to the HJBI is unique in the polynomial-growth class. In particular, if Isaacs’ condition holds so that $V=\widehat V$, then this common value is the unique viscosity solution.
\end{theorem}

The proof relies on the doubling-of-variables technique and is provided in Appendix~\ref{proo:comparison}.

\begin{remark}
The comparison principle holds under Assumption~\ref{assu:B} because:
\begin{enumerate}
\item $H$ is continuous in $(x,p)$, locally Lipschitz in $x$, with polynomial/quadratic growth (Assumption~\ref{assu:B}(2)), ensuring Ishii--Lions stability (Ishii and Lions~\cite{IshiiLions1990}, and Crandall, Ishii, and Lions~\cite{CrandallIshiiLions1992}).
\item Polynomial growth bounds (Assumption~\ref{assu:B}(3)) provide barriers for the coercive penalization at infinity. If preferred, one can localize on bounded domains and let the radius $\to\infty$ instead of using the $\zeta$-term (see Appendix~\ref{proo:comparison}).
\item The state-constraint boundary at $v=0$ (Assumption~\ref{assu:B}(4)) is handled via constrained semijets on the closed set, eliminating boundary terms in the comparison argument. The constrained viscosity framework ensures that test functions respect the state 
constraint at $v=0$, and the doubling of variables is performed only on the interior 
of the domain where both $U$ and $W$ are tested with smooth functions.
\end{enumerate}
\end{remark}

\begin{remark}
By the comparison principle, viscosity solutions are unique in the polynomial-growth class. Since \Cref{theo:visc}(1)-(3) identify $V$ as a supersolution and $\widehat V$ as a viscosity subsolution with the same terminal condition, we obtain $\widehat V \le V$. If, in addition, Isaacs' condition holds as in \Cref{theo:visc}(3), then $\widehat V = V$, and the value function is the unique viscosity solution to the HJBI.
\end{remark}

\subsubsection{Existence for the robust HJBI}\label{sect:existence}
We work on the state space $[0,T]\times\mathbb{R}\times\mathbb{R}_+$ with the state-constraint boundary at $v=0$, as in \cref{sect:comparison}. Let $H$ denote the Isaacs Hamiltonian defined in \cref{eq:IsaacsH}. Eliminating adversarial distortions via the KL dual yields the convex quadratic terms $\lambda_m p_m^2 + \lambda_v p_v^2$ in $H$.

As in Proposition~\ref{prop:HJBI}, the robust HJBI equation is
\[-\partial_t V(t,x) + H\big(x,\nabla_x V(t,x)\big) \,=\, 0, \qquad (t,x)\in [0,T)\times(\mathbb{R}\times\mathbb{R}_+),\]
with terminal condition $V(T,x)=g(x)$.

We retain the structural and growth conditions from Assumption~\ref{assu:B} for $H$. For control domains, we use one of the following options:
\begin{enumerate}
    \item Compact controls (default): the control sets $\mathcal U,\mathcal P$ are compact intervals (Assumption~\ref{assu:A}(1)).
    \item Unbounded but coercive: replace compactness by the following assumption.
\end{enumerate}

\begin{assumption}[coercive running cost]\label{assu:C}
For each fixed $(t,m,v,\theta,\xi)$, the running cost \newline $\ell(t,m,v,u,\pi,\theta,\xi)$ is coercive in $(u,\pi)$, \emph{i.e.,} $\ell(t,m,v,u,\pi,\theta,\xi)\to+\infty$ as $\lVert (u,\pi) \rVert \to \infty$.
\end{assumption}

\begin{theorem}[existence of a viscosity solution]\label{theo:existence}
Under Assumption~\ref{assu:B} and either compact controls (Assumption~\ref{assu:A}(1)) or Assumption~\ref{assu:C}, there exists a continuous viscosity solution $V$ to the HJBI with terminal condition $g$, with at most polynomial growth. Moreover, if Isaacs' condition holds so that the Isaacs Hamiltonian $H$ is well-defined, then $V$ coincides with the robust control value function defined via the DPP.
\end{theorem}

The existence is established via the convergence of a monotone approximation scheme (see Appendix~\ref{proo:existence}).

\subsection{Verification theorem and Riccati equation derivation}
We continue under the previous setting. The Isaacs Hamiltonian $H$ is the one in \cref{eq:IsaacsH}. Control sets $\mathcal U,\mathcal P$ are compact (Assumption~\ref{assu:A}(1)). Structural and growth assumptions are those in Assumption~\ref{assu:B}.

\subsubsection{Verification theorem for the robust HJBI}\label{sect:veri}
\begin{theorem}[verification statement]\label{theo:veri}
Let $V\in C([0,T]\times(\mathbb{R}\times\mathbb{R}_+))$ with at most polynomial growth satisfy, in the viscosity sense,
\[-\partial_t V(t,x) + H\big(x,\nabla_x V(t,x)\big) = 0 \quad \text{on } [0,T)\times(\mathbb{R}\times\mathbb{R}_+), \qquad V(T,x)=g(x).\]
Assume Isaacs' condition holds so that the Isaacs Hamiltonian is well-defined (see \Cref{theo:visc}(3)), and that measurable minimizers exist for the Hamiltonian, guaranteed by compactness of the action sets and continuity in the controls. Then $V$ coincides with the robust value function, and the feedback controls that minimize the Hamiltonian are optimal for the robust problem.
\end{theorem}

The proof follows verification arguments using the DPP and comparison principle. Details are in Appendix~\ref{proo:veri}.

\begin{remark}[explicit selectors]
When $w_2(u)=\bar w_2 + \kappa u$ and no saturation occurs at the control bounds, the first-order conditions for the minimization in $H$ yield
\[2 R_u\, u^* + \eta\,\partial_m V + \kappa v = 0, \qquad 2 R\, \pi^* - \chi\, \partial_v V = 0,\] hence
\begin{equation}\label{eq:minimizers}
u^* = -\frac{\eta\,\partial_m V + \kappa v}{2 R_u}, \qquad \pi^* = \frac{\chi\,\partial_v V}{2 R}.
\end{equation}
Projection onto $\mathcal U$ and $\mathcal P$ enforces the bounds.

Fix $(t,x)$ and write $q_m = \partial_m V(t,x)$ and $q_v = \partial_v V(t,x)$ for the adjoint variables (co-states). Since the adversary’s model distortions are penalized by KL, their instantaneous cost is quadratic in $\theta$ and $\xi$, normalized as $\frac{1}{4\lambda}$ times the square. In the HJBI, this yields a pointwise optimization of a linear term minus that quadratic penalty. By the Legendre transform (\emph{e.g.,} Bauschke and Combettes~\cite{BauschkeCombettes2017}, Definition 13.1),
\[\sup_{z\in\mathbb{R}}\Big\{ z\,q - \tfrac{1}{4\lambda}\,z^2 \Big\} = \lambda\,q^2, \quad\text{with maximizer}\quad z^* = 2\lambda\,q.\]
The coefficient $\tfrac{1}{4\lambda}$ arises from dualizing the relative-entropy budget in \cref{eq:robust}. The normalization is chosen so that the maximizer takes the form $z^* = 2\lambda q$. Thus, we have the following KL-dual convention:
\begin{equation}\label{eq:kl_dual}
\sup_{\theta\in\mathbb{R}}\Big\{ \theta\,\partial_m V - \tfrac{1}{4\lambda_m}\theta^2 \Big\} = \lambda_m\big(\partial_m V\big)^2, \qquad \sup_{\xi\in\mathbb{R}}\Big\{ \xi\,\partial_v V - \tfrac{1}{4\lambda_v}\xi^2 \Big\} = \lambda_v\big(\partial_v V\big)^2.
\end{equation}

\noindent Under \cref{eq:kl_dual}, the adversary solves the pointwise problems
\[\sup_{\theta\in\mathbb{R}}\Big\{ \theta\,q_m - \tfrac{1}{4\lambda_m}\theta^2 \Big\}, \qquad \sup_{\xi\in\mathbb{R}}\Big\{ \xi\,q_v - \tfrac{1}{4\lambda_v}\xi^2 \Big\}.\]
Each objective is strictly concave. Differentiating gives the unique maximizers
\[\theta^* = 2\lambda_m\,q_m, \qquad \xi^* = 2\lambda_v\,q_v.\]
Equivalently, completing the square shows
\[\theta\,q_m - \frac{1}{4\lambda_m}\theta^2 = \lambda_m q_m^2 - \frac{1}{4\lambda_m}\big(\theta - 2\lambda_m q_m\big)^2,\]
so the maximum value is $\lambda_m q_m^2$ (and analogously $\lambda_v q_v^2$ for $\xi$). Substituting the maximizers adds the terms $\lambda_m(\partial_m V)^2 + \lambda_v(\partial_v V)^2$ to the Isaacs Hamiltonian. The coefficients $\tfrac{1}{4\lambda_m}$ and $\tfrac{1}{4\lambda_v}$ are chosen so that $\sup_{z}\{z\,q - \tfrac{1}{4\lambda}z^2\}=\lambda q^2$ and $z^*=2\lambda q$ hold.
\end{remark}

\subsubsection{Existence of saddle points}\label{sect:saddle}
We continue with the same framework. In \Cref{theo:saddle}, we state that for our deterministic robust LQ setting, the Isaacs condition holds pointwise and there exist feedback $(u^*,\pi^*,\theta^*,\xi^*)$ that forms a saddle point for both the differential game and the HJBI.

\begin{theorem}[saddle point via square completion]\label{theo:saddle}
Assume $R_u>0$, $R>0$, $\lambda_m>0$, $\lambda_v>0$, and that $\mathcal U$ and $\mathcal P$ are compact convex intervals. Then for every $(t,x)$ with
$p \,=\, \nabla_x V(t,x) \,=\, (\partial_m V(t,x),\, \partial_v V(t,x))$, we have:
\begin{enumerate}
\item the min-max over $(u,\pi)\in\mathcal U\times\mathcal P$ and the max over $(\theta,\xi)\in\mathbb R^2$ commute (Isaacs holds pointwise for $H$), and
\item there exist measurable feedback $(u^*,\pi^*,\theta^*,\xi^*)$ forming a saddle point for the differential game and for the HJBI.
\end{enumerate}
Moreover, the optimal feedback coincides with the first-order minimizers/maximizers from \cref{sect:veri}:
\[(u_t^*,\pi_t^*,\theta_t^*,\xi_t^*) \,=\, \Big( \Pi_{\mathcal U}\!\big(-\tfrac{\eta\,\partial_m V + \kappa v}{2R_u}\big),\; \Pi_{\mathcal P}\!\big(\tfrac{\chi\,\partial_v V}{2R}\big),\; 2\lambda_m\,\partial_m V,\; 2\lambda_v\,\partial_v V \Big)\Big|_{(t,X_t)}.\]
Projection onto $\mathcal U$ and $\mathcal P$ enforces the bounds.
\end{theorem}

\begin{proof}
Consider the objective function inside the curly braces in \cref{eq:IsaacsH_1}. Square completion in $(\theta,\xi)$ combines the linear drift terms with the quadratic penalties in $\ell$ to yield adversarial maximizers with achieved value $\lambda_m p_m^2 + \lambda_v p_v^2$. Similarly, square completion in $(u,\pi)$ yields the unconstrained minimizers \cref{eq:minimizers}, while metric projections onto $\mathcal U$ and $\mathcal P$ enforce bounds. The mapping is continuous, strictly convex in $(u,\pi)$ and strictly concave in $(\theta,\xi)$. Since the assumption $\bar w_2 + \kappa\, u_{\min} > 0$ (\cref{sect:model}) ensures $w_2(u) > 0$ for all $u \in \mathcal{U}$, the objective remains strictly convex in $u$. With $\mathcal U\times\mathcal P$ compact and convex, Sion's~\cite{sion1958} minimax theorem applies: if $X$ is compact and convex, $Y$ is convex, and $f:X\times Y\to\mathbb R$ is convex and lower semicontinuous in $x\in X$ and concave and upper semicontinuous in $y\in Y$, then
\[ \min_{x\in X} \sup_{y\in Y} f(x,y) \;=\; \sup_{y\in Y} \min_{x\in X} f(x,y), \]
and when one side is compact and the other convex with the respective semicontinuity, the extrema are attained, yielding a saddle point. Applied here with $X=\mathcal U\times\mathcal P$, $Y=\mathbb R^2$, and the objective function continuous, strictly convex in $(u,\pi)$ and strictly concave in $(\theta,\xi)$, we obtain pointwise Isaacs equality for $H$ and a saddle point at the feedback above when $p=\nabla_x V$. Admissibility follows from polynomial growth and compactness, and optimality follows from \Cref{theo:veri} together with the DPP and the comparison principle.
\end{proof}

\begin{remark}[optimality of projected controls]
Since the objective function in the Isaacs Hamiltonian \cref{eq:IsaacsH_1} is strictly convex in $(u,\pi)$ and $\mathcal{U} \times \mathcal{P}$ is convex and compact, the constrained minimum equals the metric projection of the unconstrained minimizer onto the admissible set. This justifies the computational strategy in \cref{sect:simul}: solving the unconstrained first-order conditions and projecting onto $\mathcal{U}$ and $\mathcal{P}$ yields the optimal feedback.
\end{remark}

\begin{remark}[Riccati specialization]
If $V(t,x)$ is quadratic in $(m,v)$, the feedback is linear and the HJBI \cref{eq:HJBI} reduces to coupled Riccati equations in time. This observation motivates the quadratic ansatz \cref{eq:ansatz}.
\end{remark}

\begin{remark}[state constraint]
The state-constraint boundary at $v=0$ is treated in the viscosity sense as in \cref{sect:existence}. Projections $\Pi_{\mathcal U}$ and $\Pi_{\mathcal P}$ ensure admissibility on compact action sets.
\end{remark}

\subsubsection{Quadratic value function ansatz}\label{sect:quadratic}
We proceed under the same assumptions as in \cref{sect:saddle}. Existence of saddle points and the pointwise Isaacs property are given by \Cref{theo:saddle}. While the viscosity framework guarantees existence and uniqueness in general, the LQ structure allows us to construct the solution explicitly via a quadratic ansatz, which we now pursue.

\paragraph{Quadratic candidate.}
On the interior region (\emph{i.e.}, away from the projection boundaries so that $u=-\tfrac{\eta \,\partial_m V + \kappa v}{2R_u}$ and $\pi=\tfrac{\chi\,\partial_v V}{2R}$), consider the quadratic ansatz
\begin{equation}\label{eq:ansatz}
V(t,m,v)=a_0(t)+a_1(t)m+a_2(t)v+a_{11}(t)m^2+a_{12}(t)mv+a_{22}(t)v^2,
\end{equation}
so that $\partial_m V=a_1+2a_{11}m+a_{12}v$ and $\partial_v V=a_2+a_{12}m+2a_{22}v$, with terminal conditions matching $G_mm^2+G_vv$ at $t=T$ and thus
\begin{equation}\label{eq:app_terminal}
a_{2}(T)=G_v, \;a_{11}(T)=G_m, \quad a_0(T)=a_1(T)=a_{12}(T)=a_{22}(T)=0.
\end{equation}

\paragraph{Riccati ODE system for the quadratic ansatz.}
Let $\nabla V = (\partial_m V,\partial_v V)$. Using the Isaacs Hamiltonian in \cref{eq:IsaacsH_1} and the KL-dual convention \cref{eq:kl_dual}, adversarial maximizers and control minimizers coincide with the selectors in \cref{eq:minimizers} and \Cref{theo:saddle}:
\begin{equation}\label{eq:selectors}
\theta^* = 2\lambda_m\,\partial_m V,\quad \xi^* = 2\lambda_v\,\partial_v V,\quad \Big(u^*(x),\,\pi^*(x)\Big) = \Big(\Pi_{\mathcal U}\!\big( -\tfrac{\eta\,\partial_m V + \kappa v}{2R_u} \big),\; \Pi_{\mathcal P}\!\big( \tfrac{\chi\,\partial_v V}{2R} \big)\Big),
\end{equation}
where $x=(m,v)$ and with variance drift that includes $-\chi\,\pi$.

Let $\Sigma^2:=\sigma_L^2+\sigma_c^2$. Plugging these selectors into the HJBI $-\partial_t V + H(x,\nabla V)=0$ yields a polynomial identity in $(m,v)$. Equating coefficients of like monomials in $(m,v)$ yields a coupled Riccati-type ODE system for $\{a_i(\cdot)\}$ on $[0,T]$ under the quadratic ansatz. The full system of six coupled Riccati ODEs is as follows (see Appendix~\ref{appe:D}).
\begin{align}
 \dot{a}_{0}  &= \;\Sigma^2 a_2 + \Big(\lambda_m - \tfrac{\eta^2}{4R_u}\Big) a_1^2 + \Big(\lambda_v - \tfrac{\chi^2}{4R}\Big) a_2^2 \nonumber ,\\
 \dot{a}_{1}  &= \;\Sigma^2 a_{12} + \Big(4\lambda_m - \tfrac{\eta^2}{R_u}\Big) a_1 a_{11} + \Big(2\lambda_v - \tfrac{\chi^2}{2R}\Big) a_2 a_{12}\nonumber ,\\
 \dot{a}_{2}  &= \;\bar{w}_2 - 2\beta a_2 + 2\Sigma^2 a_{22} + \Big(2\lambda_m - \tfrac{\eta^2}{2R_u}\Big) a_1 a_{12} + \Big(4\lambda_v - \tfrac{\chi^2}{R}\Big) a_2 a_{22} - \tfrac{\eta\kappa}{2R_u} a_1 \label{eq:Riccati},\\
 \dot{a}_{11} &= \;w_1 + \Big(4\lambda_m - \tfrac{\eta^2}{R_u}\Big) a_{11}^2 + \Big(\lambda_v - \tfrac{\chi^2}{4R}\Big) a_{12}^2\nonumber ,\\
 \dot{a}_{12} &= \;-2\beta a_{12} - \tfrac{\eta\kappa}{R_u} a_{11} + \Big(4\lambda_m - \tfrac{\eta^2}{R_u}\Big) a_{11} a_{12} + \Big(4\lambda_v - \tfrac{\chi^2}{R}\Big) a_{12} a_{22} \nonumber,\\
 \dot{a}_{22} &= \;-4\beta a_{22} - \tfrac{\kappa^2}{4R_u} - \tfrac{\eta\kappa}{2R_u} a_{12} + \Big(\lambda_m - \tfrac{\eta^2}{4R_u}\Big) a_{12}^2 + \Big(4\lambda_v - \tfrac{\chi^2}{R}\Big) a_{22}^2\nonumber.
\end{align}

\begin{remark}
The system governs the interior (unconstrained) regime. When projections are active, the selectors in \cref{eq:selectors} become piecewise, and the coefficient dynamics are piecewise with switching times determined by the projection boundaries. The scaling is consistent with the KL-dual convention \cref{eq:kl_dual} in which the running penalty adds $-\tfrac{\theta^2}{4\lambda_m}$ and $-\tfrac{\xi^2}{4\lambda_v}$ to the Hamiltonian, yielding the multipliers in the system. The constant variance drift $\Sigma^2$ enters linearly and affects only the equations for $a_0,a_1,a_2$.
\end{remark}

\begin{theorem}[quadratic verification by cross-reference]
Let $V$ be given by the quadratic ansatz \cref{eq:ansatz} with coefficients ${a_i(\cdot)}$ on $[0,T]$ and terminal condition \cref{eq:app_terminal}. Suppose the Riccati system in \cref{eq:Riccati} admits a $C^1$ solution on $[0,T]$ with at most polynomial growth and that the assumptions hold (compact convex action sets, measurable selectors, Lipschitz dynamics, and Isaacs condition). Define feedback $(u^*,\pi^*;\theta^*,\xi^*)$ as in \cref{eq:selectors} with projections onto the admissible sets. Then $V$ is a viscosity solution of the HJBI with terminal condition $V(T,m,v)=G_mm^2+G_vv$, and
\[V(t,x) \,=\, \inf_{u,\pi}\,\sup_{\theta,\xi} J_{t,x}(u,\pi;\theta,\xi) \,=\, \sup_{\theta,\xi}\,\inf_{u,\pi} J_{t,x}(u,\pi;\theta,\xi),\]
with $(u^*,\pi^*;\theta^*,\xi^*)$ forming a saddle point.
\end{theorem}

\begin{proof}
By \Cref{theo:veri} and the Isaacs/saddle-point result in \Cref{theo:saddle}. The selectors \cref{eq:selectors} realize the Hamiltonian extremizers. Plugging them into the HJBI yields the polynomial identity in $(m,v)$, whose coefficient matching is equivalent to deriving \cref{eq:Riccati}. Hence $V$ solves the HJBI in the viscosity sense with the stated terminal condition, and the value identity with the saddle point follows by Isaacs’ condition.
\end{proof}

\begin{proposition}[global existence and breakdown threshold]\label{prop:stability}
The Riccati system \cref{eq:Riccati} admits a unique bounded solution on $[0,T]$ for any time horizon $T > 0$ if and only if the adversary parameters satisfy the stability conditions:
\begin{equation}
4\lambda_m \le \frac{\eta^2}{R_u} \quad \text{and} \quad 4\lambda_v \le \frac{\chi^2}{R}.
\end{equation}
If these conditions are strictly satisfied, the solution remains bounded. If either condition is strictly violated (\emph{i.e.,} $4\lambda_m > \frac{\eta^2}{R_u}$ or $4\lambda_v > \frac{\chi^2}{R}$), there exists a critical horizon $T^* < \infty$ such that for any $T > T^*$, the solution to the Riccati system explodes, implying $V(0,m,v) = +\infty$ and the non-existence of a finite-cost robust control policy.
\end{proposition}

\begin{proof}
The Riccati equations for the second-order coefficients $a_{11}, a_{12}, a_{22}$ form a closed subsystem. Consider the $a_{11}$ equation:
\[\dot{a}_{11} = w_1 + \left(4\lambda_m - \tfrac{\eta^2}{R_u}\right) a_{11}^2 + \left(\lambda_v - \tfrac{\chi^2}{4R}\right) a_{12}^2.\]
Let $C_m = 4\lambda_m - \frac{\eta^2}{R_u}$. When the stability condition is violated ($C_m > 0$), the coefficient $a_{11}(t)$ is bounded from below by the solution of the comparison ODE $\dot{y} = w_1 + C_m y^2$ with $y(T) = G_m$. The solution to this ODE is of the form $y(t) \propto \tan\left(\sqrt{w_1 C_m}(T-t) + c\right)$. 

The tangent function exhibits a vertical asymptote at a finite backward time distance $T^*$ determined by the terminal value and coefficients. Consequently, for horizons $T > T^*$, the solution $a_{11}(t)$ explodes, causing the value function to become infinite. A similar argument applies to $a_{22}$ when $4\lambda_v > \frac{\chi^2}{R}$. This establishes the non-existence of finite-cost robust control when the stability conditions are violated.
\end{proof}

\begin{remark}[interior vs constrained regimes]
The Riccati ODEs in \cref{eq:Riccati} govern the interior region where projections are inactive. Under Lipschitz state dynamics and compact action sets, admissible controls generate unique absolutely continuous trajectories, and polynomial growth of $V$ ensures integrability. In extensions with noise, viscosity well-posedness of the robust HJBI follows under standard comparison hypotheses (\emph{e.g.,} continuity, properness, and structural conditions).
\end{remark}

\begin{remark}[piecewise dynamics in saturated regimes]
When a control saturates at its boundary, the structure of the Riccati system changes because the minimization (infimum over $(u, \pi)$) in \cref{eq:IsaacsH} is replaced by boundary evaluation.
\begin{enumerate}
\item When $\pi_t = \pi_{\max}$, the optimal control becomes constant, and the term $-\frac{(\partial_v V \chi)^2}{4R}$ in the optimized Hamiltonian is replaced by $R\pi_{\max}^2 - \partial_v V \chi \pi_{\max}$. This alters the Riccati coefficients in \cref{eq:Riccati}. For instance, the quadratic self-interaction term $-\frac{\chi^2}{R}a_{22}^2$ in the $\dot{a}_{22}$ equation is removed, and linear terms in $a_{22}$ are modified.
\item Analogous changes occur when $u_t$ saturates.
\end{enumerate}
While the value function $V$ remains $C^1$ across switching boundaries, solving the exact piecewise system requires tracking switching surfaces. For computational tractability, our numerical implementation in \cref{sect:simul} solves the interior Riccati system and projects the resulting controls onto $\mathcal{U}$ and $\mathcal{P}$ at each time step.
\end{remark}

\section{Simulations and robustness}\label{sect:simul}
Building on the theoretical foundations in \cref{sect:theor}, we perform simulations and assess robustness. Unless stated otherwise, the simulations in \cref{sect:simul} are implemented in two stages using the parameters in \Cref{tabl:param}.

First, we obtain the time-varying coefficients for the quadratic ansatz $V(t,m,v)$ by integrating the Riccati system \cref{eq:Riccati} backward from the terminal condition $g(m,v)=G_mm^2+G_vv$, implemented as \cref{eq:app_terminal}. We employ an implicit, L-stable (Laplace transform stable) Runge-Kutta method of Radau IIA type to handle the Riccati equations.

Second, we propagate the forward dynamics using an explicit Euler scheme on a uniform grid with step $\Delta t$. At each step, we compute the unconstrained selectors using the value gradients $\partial_m V$ and $\partial_v V$, and project the controls onto their admissible sets $\mathcal U$ and $\mathcal P$. The variance is strictly enforced to be non-negative by flooring $v_t$ at zero. Adversarial distortions $(\theta_t, \xi_t)$ are computed via their KL-dual selectors, given by $\theta^*_t = 2\lambda_m \partial_m V$ and $\xi^*_t = 2\lambda_v \partial_v V$.

\begin{table}[htbp!]\centering
\begin{tabular}{@{}lll@{}}
\toprule
Parameters & Values & Description \\
\midrule
$(\beta, \eta, \chi, \sigma_L, \sigma_c)$ & (0.25, 0.8, 0.5, 0.4, 0.3) & System parameters \\
$(w_1, \bar{w}_2, \kappa, R_u, R)$ & (0.1, 0.5, 0.05, 0.5, 0.25) & Cost parameters \\
$(G_m, G_v)$ & (0.5, 0.5) & Terminal cost weights \\
$(\lambda_m, \lambda_v)$ & (0.02, 0.02) & Adversary strength parameters \\
$(u_{\min}, u_{\max}, \pi_{\max})$ & (-1.0, 1.0, 10.0) & Control bounds \\
$(T, \Delta t)$ & (10.0, 0.001) & Time parameters \\
$(m_0, v_0)$ & (0.5, 1.0) & Initial conditions \\
\bottomrule
\end{tabular}
\caption{Simulation parameters of the baseline model.}\label{tabl:param}
\end{table}

The baseline parameters in \Cref{tabl:param} are chosen to ensure the robust control problem is well-posed. A bounded solution to the Riccati system requires the stabilizing control-cost terms to exceed the destabilizing adversarial terms. Our parameters strictly satisfy these conditions: $\frac{\eta^2}{R_u} = 1.28 > 4\lambda_m = 0.08$ for the mean and $\frac{\chi^2}{R} = 1.00 > 4\lambda_v = 0.08$ for the variance, guaranteeing a stable interior solution and allowing us to analyze the system's response as $\lambda$ increases.

\begin{remark}[robustness-breakdown threshold]\label{rema:break}
The stability conditions determine the \emph{robustness-breakdown} threshold (see Proposition~\ref{prop:stability}). If the adversary's strength becomes too large ($4\lambda_m > \frac{\eta^2}{R_u}$ or $4\lambda_v > \frac{\chi^2}{R}$), the coefficient on the corresponding quadratic term in the Riccati system \cref{eq:Riccati} ($a_{11}^2$ in the $\dot{a}_{11}$ equation or $a_{22}^2$ in the $\dot{a}_{22}$ equation) becomes positive. Solving backward from a finite terminal condition, such a solution explodes to infinity in finite time, implying that no finite-cost optimal policy exists. Thus, the breakdown threshold is linked to the critical value $\lambda^*$, where the coefficient changes sign from negative to positive, and the robust control problem ceases to have a well-posed solution.
\end{remark}

\begin{remark}[baseline parameters]\label{rema:param}
Beyond stability, the baseline parameters shape the Riccati dynamics and optimal feedback structure. The control effectiveness ratio $\frac{\eta^2}{R_u} = 1.28 > \frac{\chi^2}{R} = 1.00$ implies that the mean channel dominates in the Riccati coefficients, explaining the system's robustness to mean distortions but vulnerability to variance ambiguity. The cost weights $\bar{w}_2 = 0.5 > w_1 = 0.1$ weight variance reduction more heavily, which amplifies the $a_2$ and $a_{22}$ coefficients and drives active coordination between instruments. The terminal weights $G_m = G_v = 0.5$ impose symmetric penalties on final deviations. The mean-reversion $\beta=0.25$ yields an autonomous variance decay rate $2\beta = 0.5$, which balances against the noise injection $\Sigma^2 = \sigma_L^2 + \sigma_c^2 = 0.25$. The coupling $\kappa=0.05$ introduces cross-terms in the Hamiltonian linking $v_t$ to the policy rate optimization, but remains small relative to $\bar{w}_2$ to preserve near-separability. Finally, the baseline adversary strengths $\lambda_m=\lambda_v=0.02$ provide substantial headroom below the breakdown thresholds $\lambda_m^* = 0.32$ and $\lambda_v^* = 0.25$, enabling exploration of the transition to loss of control.
\end{remark}

\begin{remark}[over-monitoring and state constraints]\label{rema:over}
A structural consequence of the unconstrained formulation is \emph{over-monitoring}, which arises because the interior Riccati solution yields a global quadratic value function without enforcing the state constraint at $v=0$. A fully constrained formulation would require solving the HJBI equation with state-constraint boundary conditions in the viscosity sense of Soner~\cite{Soner1986II}, introducing regime-dependent dynamics that reduce $\pi_t$ once $v_t$ reaches zero. However, because the constraint introduces endogenous regime switching, the backward Riccati system and forward state dynamics must be solved jointly, necessitating numerical methods. We retain our unconstrained formulation to preserve analytical tractability, as it yields explicit feedback laws that transparently illustrate the coupling between monetary and supervisory policies.\footnote{CBs typically conduct bank monitoring even when a bank's current liquidity conditions are sound, and our formulation is consistent with this institutional practice.}
\end{remark}

\begin{remark}[scope of over-monitoring effects]\label{rema:over_scope}
Over-monitoring primarily inflates the total cost $J$ and control saturation metrics. State trajectories $(m_t, v_t)$ and peak adversarial distortions are largely unaffected, as the variance floor at $v_t = 0$ effectively captures the stabilized dynamics. Robustness-breakdown thresholds (determined by Riccati stability) are also unaffected. However, the cost impact is parameter-dependent: high-$\chi$ regimes, where the system reaches the boundary earlier and remains there longer, are disproportionately affected. Consequently, while the structural phase transitions in the loss-of-control analysis (\Cref{figu:loss}) remain valid, the iso-cost contours in high-$\chi$ regions reflect this inefficiency.
\end{remark}

\subsection{Path simulations}\label{sect:4.1}
We simulate finite-horizon paths of $(m_t, v_t, u_t, \pi_t)$ under baseline parameters for different levels of adversary strength, as in \Cref{algo:paths}.

\begin{algorithm}[htbp!]
\caption{System dynamics (\Cref{figu:path})}\label{algo:paths}
\begin{algorithmic}[1]
\State \textbf{Input:} Parameter set $p$, horizon $T$, step $\Delta t$, grid $t_n = n\Delta t$.
\State \textbf{Output:} Trajectories of $(m, v, u, \pi)$ and distortions $(\theta, \xi)$.
\Function{Simulate}{$p$}
    \State \textbf{Backward pass:}
    \State Solve Riccati system (\cref{eq:Riccati}) for $a(t)$ on grid using Radau IIA.    
    \State \textbf{Forward pass:}
    \State Initialize $(m_0, v_0)$.
    \For{$n = 0, \ldots, N-1$}
        \State Evaluate gradients $\partial_m V, \partial_v V$ using $a(t_n), m_n, v_n$.
        \State Compute distortions: $\theta_n = 2\lambda_m \partial_m V$, $\xi_n = 2\lambda_v \partial_v V$.
        \State Compute unconstrained controls $(u^{\text{unc}}, \pi^{\text{unc}})$ via \cref{eq:selectors}.
        \State \textbf{Project:} $u_n \leftarrow \text{clip}(u^{\text{unc}}, u_{\min}, u_{\max})$, $\pi_n \leftarrow \text{clip}(\pi^{\text{unc}}, 0, \pi_{\max})$.
        \State \textbf{Update:} 
        \State $\quad m_{n+1} \leftarrow m_n + [ \eta u_n + \theta_n ] \Delta t$
        \State $\quad v_{n+1} \leftarrow \max\big(0, v_n + [ -2\beta v_n + \Sigma^2 + \xi_n - \chi \pi_n ] \Delta t\big)$
    \EndFor
    \State \Return Trajectories $\{(m_n, v_n, u_n, \pi_n, \theta_n, \xi_n)\}_{n=0}^N$.
\EndFunction
\end{algorithmic}
\end{algorithm}

\begin{figure}[htbp!] \centering\includegraphics[width=0.7\linewidth]{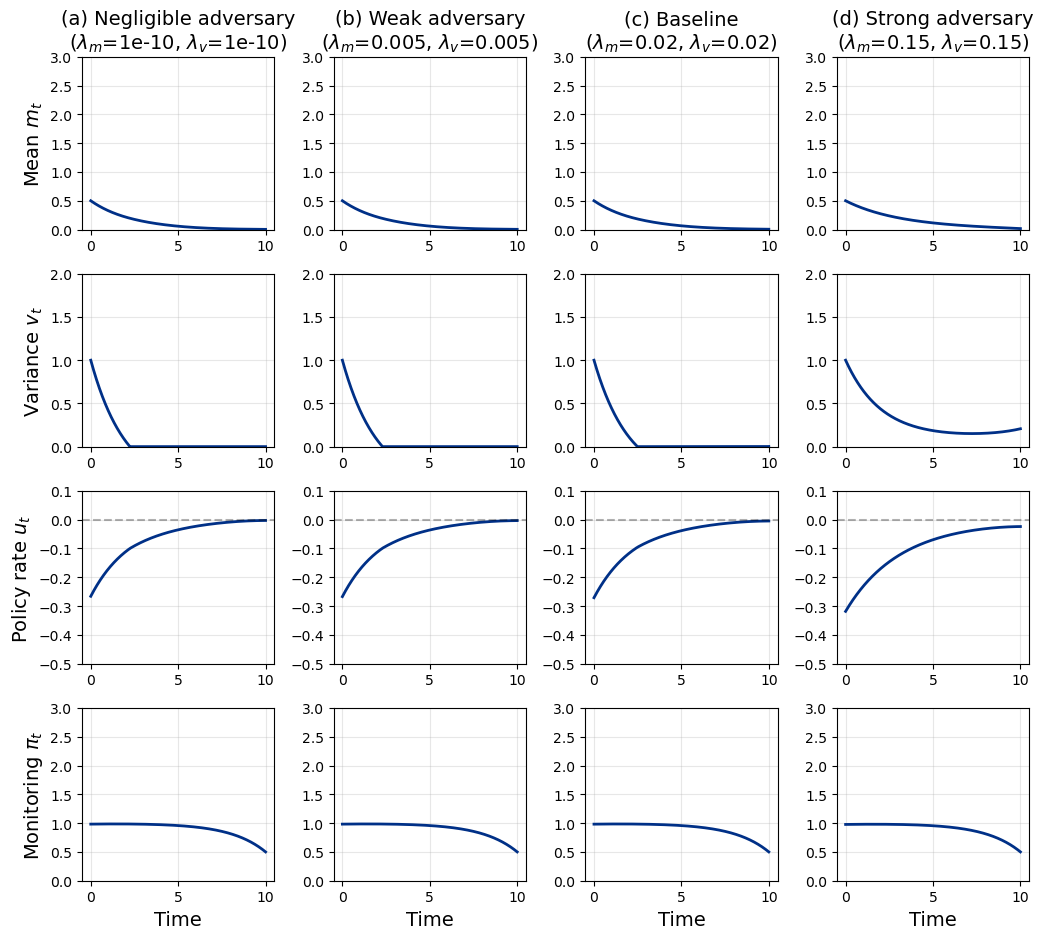}\caption{System dynamics under different levels of adversary strength. Panels show trajectories for mean liquidity $m_t$, variance $v_t$, policy rate $u_t$, and monitoring intensity $\pi_t$. As $\lambda$ increases, $v_T$ is pushed upward, settling at a non-zero steady state in the strong-adversary case. Note in panels (a) to (c) that monitoring $\pi_t$ remains positive even after variance $v_t$ reaches zero, illustrating the over-monitoring discussed in Remarks~\ref{rema:over} and~\ref{rema:over_scope}.}\label{figu:path}
\end{figure}

We call \textsc{Simulate} with four adversary strength scenarios: Negligible ($\lambda_m=\lambda_v=10^{-10}$), Weak ($\lambda_m=\lambda_v=0.005$), Baseline ($\lambda_m=\lambda_v=0.02$), and Strong ($\lambda_m=\lambda_v=0.15$), producing the trajectories shown in \Cref{figu:path}.

Initially, the mean $m_0 = 0.5$ reflects a liquidity shortage in the banking sector. The policy rate $u_t$ starts around $-0.3$, an accommodative stance consistent with addressing this shortage. The variance $v_0 = 1.0$ indicates heterogeneity in liquidity positions across banks. The monitoring intensity $\pi_0$ begins close to $1$, reflecting the CB’s heightened monitoring in response to this initial stress ($v_0=1.0$).

In \Cref{figu:path}, the mean liquidity $m_t$ is steered from its initial value of $0.5$ toward its terminal value $m_T$, driven by the policy rate $u_t$ which relaxes from its initial aggressive stance. The monitoring intensity $\pi_t$ decreases from its high initial value, steering the variance downward. As adversary strength $\lambda$ increases, the final states for both mean and variance $(m_T, v_T)$ become progressively higher. This effect is most clear in \Cref{figu:path}(d), where the variance settles at a non-zero value, consistent with the optimal equilibrium described in Remark~\ref{rema:terminal}. With a strong adversary, the total upward pressure on variance ($\Sigma^2 + \xi_t^*$) exceeds the CB's downward control force ($\chi\pi_t^*$) near the terminal time. Consequently, $v_t$ settles at an equilibrium where the optimal monitoring condition holds ($2R\pi_t^* = \chi \partial_v V$) and the variance drift balances to zero ($2\beta v_t = \Sigma^2 + \xi_t^* - \chi \pi_t^*$).

\subsection{Adversary strength analysis}\label{sect:4.2}
We sweep the adversary strength parameters $(\lambda_m, \lambda_v)$ and track indicators of control effectiveness, as in \Cref{algo:adv_1}. $\lambda_m$ and $\lambda_v$ govern the adversary's capacity to distort the mean and variance channels, respectively. For any parameter tuple, we simulate state and control paths $(m_t, v_t, u_t, \pi_t)$ over a finite horizon.

\begin{algorithm}[htbp!]
\caption{Adversary strength analysis (\Cref{figu:adv_1})}\label{algo:adv_1}
\begin{algorithmic}[1]
\State \textbf{Input:} Baseline parameters $p_0$, grid size $N$.
\State \textbf{Output:} Cost and control metrics versus $\lambda$.
\State \textbf{Part 1: Symmetric analysis}
\For{$\lambda \in \text{linspace}(0, 0.2, N)$}
    \State Update: $p \leftarrow p_0$ with $\lambda_m = \lambda_v = \lambda$
    \State Simulate and store metrics
\EndFor
\State \textbf{Part 2: Asymmetric analysis}  
\State Test cases: $(\lambda_m, \lambda_v) \in \{(0.001, 0.1), (0.001, 0.2), (0.1, 0.001), (0.2, 0.001)\}$
\For{each $(\lambda_m, \lambda_v)$ pair}
    \State Update: $p \leftarrow p_0$ with specified $\lambda_m, \lambda_v$
    \State Simulate and store metrics
\EndFor
\State Generate \Cref{figu:adv_1} plots from collected metrics
\end{algorithmic}
\end{algorithm}

For visualization, we summarize performance across different levels of adversary strength by sweeping $(\lambda_m,\lambda_v)$ over orders of magnitude and plotting: total cost $J$, average controls $(\bar{u},\bar{\pi})$, and peak adversarial distortions $(\max_t |\theta_t|,\max_t |\xi_t|)$. $u_t$ remains non-positive, and its absolute value is not taken.

Our analysis of the $\lambda$ parameters is conducted within the model's stable robustness region. The coupled Riccati system in \cref{eq:Riccati} only admits a finite, bounded solution if the stabilizing control terms are stronger than the destabilizing adversarial terms (Remark~\ref{rema:break}). For our baseline parameters, this requires $\lambda_m < 0.32$ and $\lambda_v < 0.25$. We therefore restrict our $\lambda$ sweeps to the range $[0, 0.2]$ to analyze the system's behavior within this stable region, avoiding the breakdown of the Riccati solution.

\begin{figure}[htbp!] \centering\includegraphics[width=1.00\linewidth]{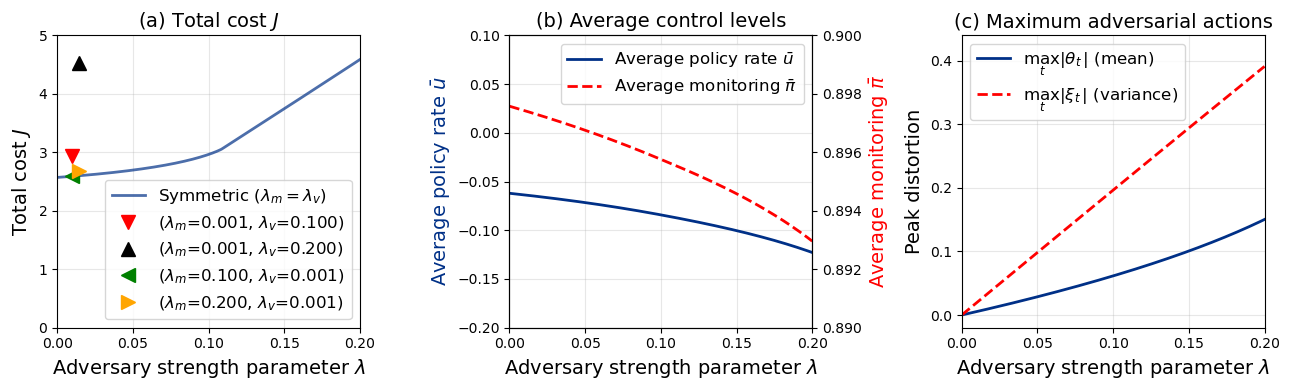}\caption{Adversary strength analysis sweeping $\lambda \in [0, 0.2]$. (a) Total cost $J$, (b) average control levels, and (c) peak adversarial distortions. Within this stable region, $J$ is primarily driven by $\lambda_v$, while no control saturation occurs (see Remarks~\ref{rema:over} and~\ref{rema:over_scope}).}\label{figu:adv_1}
\end{figure}

As $\lambda$ increases, $J$ also increases. The rise is driven primarily by the growth in $\lambda_v$, rather than $\lambda_m$ (\Cref{figu:adv_1}(a)). Both $u_t$ and $\pi_t$ adjust to $\lambda$ (\Cref{figu:adv_1}(b)). Within this range of $\lambda$, neither $u_t$ nor $\pi_t$ reach its bounds (no saturation occurs). As $\lambda$ increases, both $\max_t |\theta_t|$ and $\max_t |\xi_t|$ grow (\Cref{figu:adv_1}(c)).

We further analyze the robustness--efficiency trade-off, parameterized by the adversary strength $\lambda$ (where larger $\lambda$ implies a stronger adversary and weaker robustness), as in \Cref{algo:adv_2}.

\begin{algorithm}[htbp!]
\caption{Robustness--efficiency trade-off (\Cref{figu:adv_2})}\label{algo:adv_2}
\begin{algorithmic}[1]
\State \textbf{Input:} Grids $\{\lambda_m^i\}_{i=1}^I, \{\lambda_v^j\}_{j=1}^J$, baseline parameters $p_0$.
\State \textbf{Output:} Heatmaps and trade-off curves.
\State \textbf{Part 1: Heatmaps}
\For{$i = 1 \to I$}
    \For{$j = 1 \to J$}
        \State Update parameters: $p \leftarrow p_0$ with $\lambda_m^i, \lambda_v^j$.
        \State Simulate and extract $J[i,j], \bar{u}[i,j], \bar{\pi}[i,j], v_T[i,j]$.
    \EndFor
\EndFor
\State Generate heatmaps (panels (a) to (d))
\State \textbf{Part 2: Trade-off Curves}
\State Fix $\lambda_v = 0.02$, extract $\bar{u}(\lambda_m)$ and $\bar{\pi}(\lambda_m)$ \Comment{Panel (e)}
\State Fix $\lambda_m = 0.02$, extract $\bar{u}(\lambda_v)$ and $\bar{\pi}(\lambda_v)$ \Comment{Panel (f)}
\State Plot trade-off curves
\end{algorithmic}
\end{algorithm}

For each pair $(\lambda_m,\lambda_v)$, the procedure first obtains time-varying quadratic value coefficients, and then uses these coefficients to compute optimal controls together with the associated adversarial distortions. From the forward simulation, the simulation aggregates the total cost $J$ over the horizon, the average policy rate $\bar{u}$, the average monitoring intensity $\bar{\pi}$, and the terminal variance $v_T$. Repeating this workflow on a grid in $(\lambda_m,\lambda_v)$ maps the robustness--efficiency tradeoff. This sensitivity analysis is conducted over $100 \times 100 = 10,000$ combinations of $(\lambda_m, \lambda_v) \in [0.005, 0.2]^2$ on linearly spaced grids.

\begin{figure}[htbp!]\centering\includegraphics[width=1.00\linewidth]{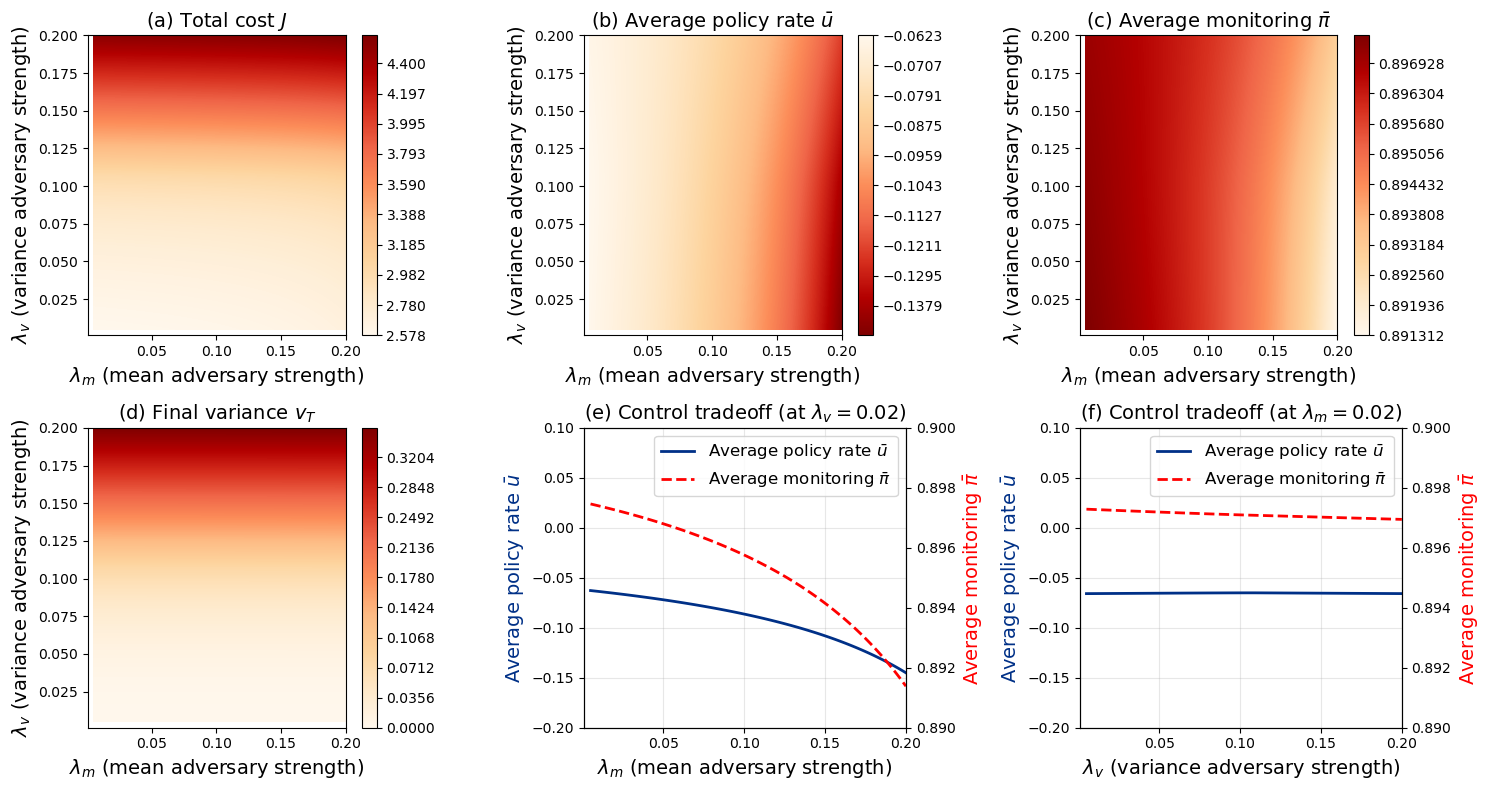}\caption{Robustness--efficiency tradeoff. (a)--(d): Heatmaps of cost, controls, and terminal variance over the $(\lambda_m, \lambda_v)$ plane. (e)--(f): Cross-sectional policy response curves fixing one adversary parameter. The heatmaps show the asymmetric policy response to $\lambda_m$ and $\lambda_v$, with $J$ increasing along $\lambda_v$ but remaining insensitive to $\lambda_m$.}\label{figu:adv_2}
\end{figure}

\Cref{figu:adv_2} summarizes the resulting robustness--efficiency landscape. The four heatmaps over $(\lambda_m, \lambda_v)$ report $J$, $\bar{u}$, $\bar{\pi}$, and $v_T$ (\Cref{figu:adv_2}(a),(b),(c),(d)). \Cref{figu:adv_2}(a) decomposes \Cref{figu:adv_1}(a). The total cost is minimized at lower $\lambda_v$ and increases as $\lambda_v$ rises, but not with $\lambda_m$ (\Cref{figu:adv_2}(a)). Both $\bar{u}$ and $\bar{\pi}$ decline with $\lambda_m$ (\Cref{figu:adv_2}(b),(c)). $v_T$ rises with $\lambda_v$ (\Cref{figu:adv_2}(d)).

The two trade-off plots show $\bar{u}$ (left axis) and $\bar{\pi}$ (right axis) (\Cref{figu:adv_2}(e),(f)). These plots decompose \Cref{figu:adv_1}(b): $\lambda_m$ varies with $\lambda_v$ fixed at $0.02$ (\Cref{figu:adv_2}(e)), and $\lambda_v$ varies with $\lambda_m$ fixed at $0.02$ (\Cref{figu:adv_2}(f)). Both $\bar{u}$ and $\bar{\pi}$ decline as $\lambda_m$ increases (\Cref{figu:adv_2}(e)), but they do not respond to increases in $\lambda_v$ (\Cref{figu:adv_2}(f)). $u_t$ does not react to $\lambda_v$ because it is the controller for the mean, and $\pi_t$ shows little response because its value is largely anchored by non-adversarial parameters ($\bar{w}_2$ and $\beta$).

Importantly, $\bar{\pi}$ decreases as $\lambda_m$ increases (\Cref{figu:adv_2}(c),(e)). When $\lambda_m$ rises, the adversarial distortion $\theta_t^* = 2\lambda_m \partial_m V$ becomes the dominant driver of instability, forcing the CB to prioritize the mean channel. This trade-off, captured by the coupled Riccati system, reduces the variance weight $w_2(u_t)$ (as $u_t$ becomes more aggressive), which in turn lowers the marginal cost of variance $\partial_v V$, which in turn reduces the optimal monitoring intensity $\pi_t^* = \frac{\chi \partial_v V}{2R}$. This may be interpreted as a resource-allocation shift at the CB.

The total cost $J$ consists essentially of $w_1m_t^2, w_2(u_t)v_t, R\pi_t^2,$ and $R_uu_t^2$. When $\lambda_v$ increases while $\lambda_m$ remains fixed, the controls $(u_t, \pi_t)$ and the mean $m_t$ stay stable, so their associated costs do not increase. However, since $\lambda_v$ increases the variance $v_t$ while $\pi_t$ remains almost unchanged, the term $w_2(u_t)v_t$ rises, thereby increasing $J$ (see Remarks~\ref{rema:over} and~\ref{rema:over_scope}).

This asymmetric sensitivity reflects a general mechanism: the system is more vulnerable to adversarial pressure on the weaker control channel. Under our baseline parameters, the mean control effectiveness $\frac{\eta^2}{R_u} = 1.28$ exceeds the variance control effectiveness $\frac{\chi^2}{R} = 1.00$, making the variance channel more exposed. Consequently, $J$ responds primarily to $\lambda_v$. The direction of this asymmetry would reverse if $\frac{\chi^2}{R} > \frac{\eta^2}{R_u}$.

\subsection{Parameter sensitivity analysis}\label{sect:4.3}
We then study how model primitives affect outcomes, as in \Cref{algo:sens}. Adversary strength parameters are held fixed at $\lambda_m = \lambda_v = 0.02$. For each parameter configuration, we compute the total cost $J$, the terminal variance $v_T$, and the saturation levels of the controls $u_t$ and $\pi_t$.

\begin{algorithm}[htbp!]
\caption{Parameter sensitivity (\Cref{figu:sens}) and control saturation analysis (\Cref{figu:satu})}\label{algo:sens}
\begin{algorithmic}[1]
\State \textbf{Input:} Baseline $p_0$ (with $\lambda_m=\lambda_v=0.02$); grids for each parameter $\vartheta \in \{\eta,\chi,\beta,\kappa,R_u,R\}$.
\State \textbf{Output:} Sensitivity curves and saturation profiles.
\For{each parameter $\vartheta$}
    \For{each value $\vartheta'$ in grid $\mathcal{G}_\vartheta$}
        \State Update parameters: $p \leftarrow p_0$ with $\vartheta \leftarrow \vartheta'$.
        \State Simulate and compute metrics: $J, v_T, \bar{u}, \bar{\pi}, S_u, S_\pi$
    \EndFor
    \State Generate panels in \Cref{figu:sens}: $J(\vartheta)$ and $v_T(\vartheta)$
    \State Generate panels in \Cref{figu:satu}: $S_u(\vartheta)$ and $S_\pi(\vartheta)$
\EndFor
\end{algorithmic}
\end{algorithm}

\begin{figure}[htbp!] \centering \includegraphics[width=0.82\linewidth]{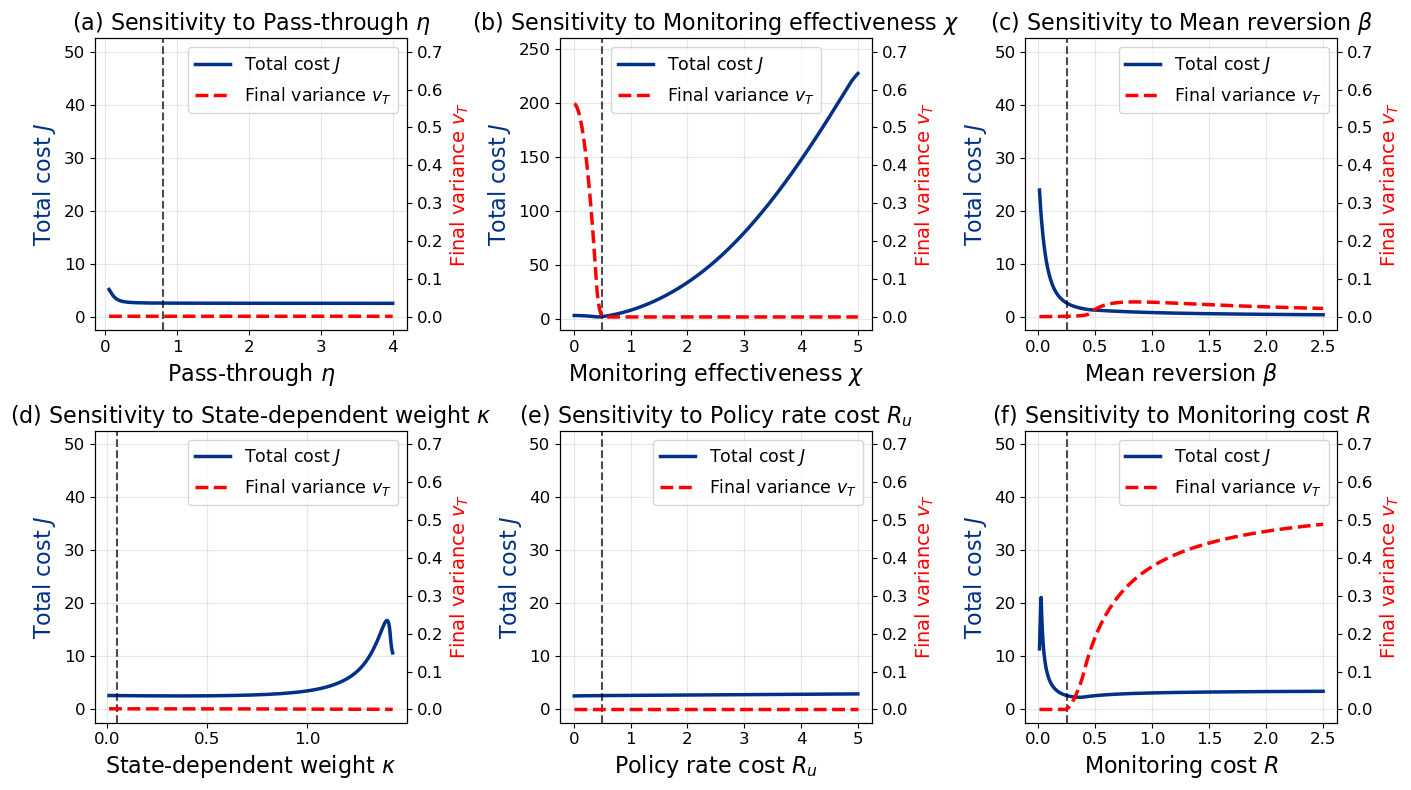}\caption{Parameter sensitivity analysis. $J$ and $v_T$ are most sensitive to $\chi$, $\beta$, and $R$. The increase in $J$ with $\chi$ (panel (b)) reflects the over-monitoring cost at the $v_t=0$ boundary (Remarks~\ref{rema:over} and~\ref{rema:over_scope}).}\label{figu:sens}
\end{figure}

In both \Cref{figu:sens,figu:satu}, the dotted vertical line in each panel indicates the baseline value of the parameter being varied. The sensitivity analysis shows that, within the tested ranges, monitoring effectiveness $\chi$, mean reversion $\beta$, and monitoring cost $R$ are the most influential parameters for both $J$ and $v_T$ (\Cref{figu:sens}(b),(c),(f)).

In particular, when $\chi$ increases, the variance drift becomes more sensitive 
to monitoring (\Cref{figu:sens}(b)). While this makes variance reduction more 
effective per unit of $\pi$, high $\chi$ drives $v_t$ to zero rapidly. At 
$v_t = 0$, the variance drift satisfies $\dot{v}_t = \Sigma^2 + \xi_t - \chi\pi_t$. 
When $\chi\pi_t > \Sigma^2 + \xi_t$, the control force exceeds noise plus 
adversarial pressure, driving variance to zero and keeping it there. Under 
baseline parameters with high $\chi$, this condition is satisfied, confirming 
that $v_t = 0$ reflects genuine control dominance.

However, once $v_t$ reaches zero, the Riccati-based feedback 
$\pi^* = \frac{\chi \partial_v V}{2R}$ continues to prescribe positive monitoring, 
but further variance reduction is impossible. The resulting over-monitoring cost 
$R\pi_t^2$ accumulates without benefit, which is the primary reason $J$ rises 
despite the improved monitoring effectiveness (see Remarks~\ref{rema:over} and~\ref{rema:over_scope}).

Across the tested ranges, the pass-through $\eta$, state-dependent weight $\kappa$, and policy rate cost $R_u$ are less influential for $J$ and $v_T$ (\Cref{figu:sens}(a),(d),(e)).

\begin{figure}[htbp!]\centering\includegraphics[width=0.8\linewidth]{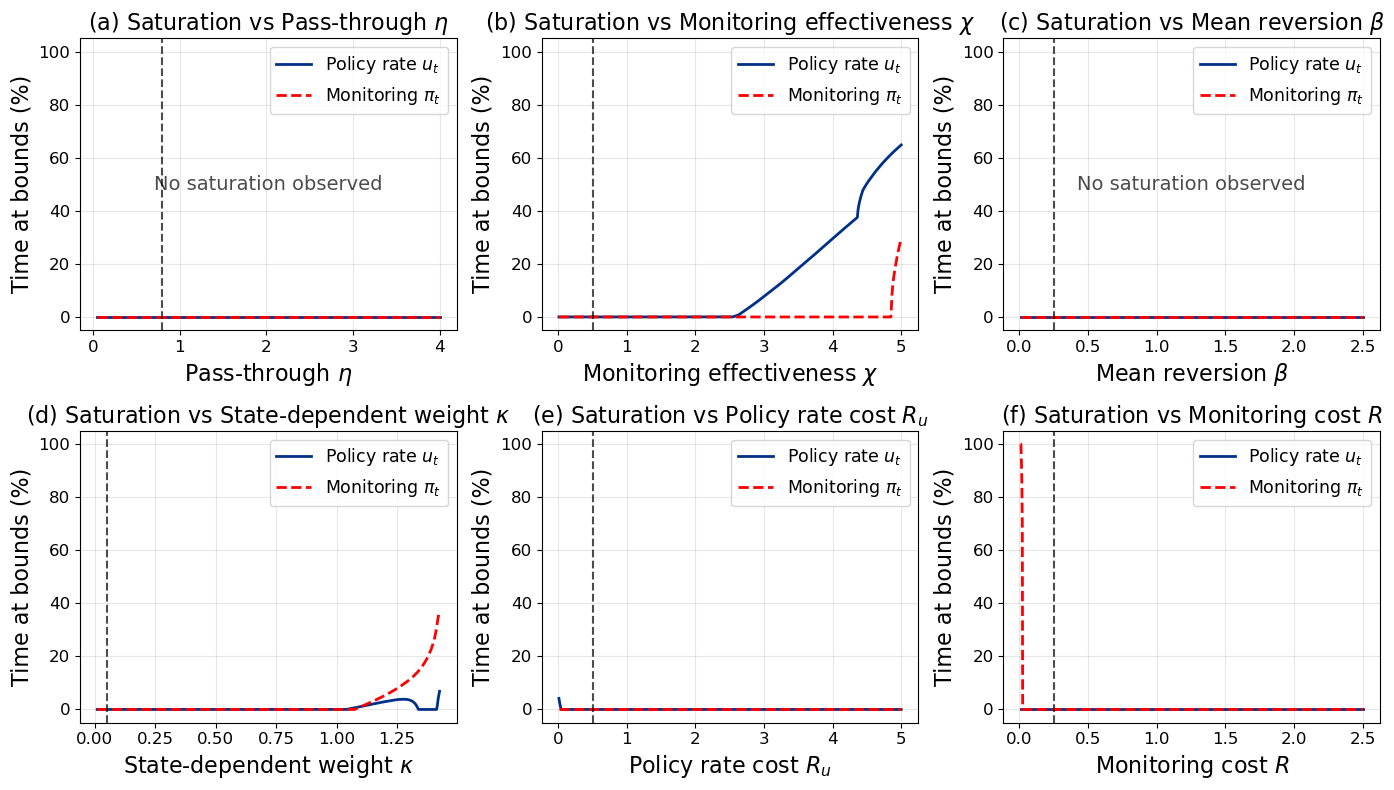}\caption{Control saturation analysis. Increasing monitoring effectiveness $\chi$ not only saturates $\pi_t$ but also drives $u_t$ to its bound.}\label{figu:satu}
\end{figure}

In the control saturation analysis, two parameters exhibit no saturation: neither $u_t$ nor $\pi_t$ reaches its bounds over the sweep (\Cref{figu:satu}(a),(c)). This indicates that the optimal controls remain interior across those ranges. Saturation of $\pi_t$ tends to arise when $\chi$ increases, $\kappa$ rises, or $R$ declines (\Cref{figu:satu}(b),(d),(f)). Saturation of $u_t$ emerges when $\chi$ increases (\Cref{figu:satu}(b)).

When $\chi$ increases, $\frac{\chi^2}{R}$ becomes stronger. The Riccati system responds by increasing the magnitude of the $a_{12}$ coupling term, which links variance $v_t$ to the mean controller $u_t$. This stronger coupling amplifies the variance contribution to the optimal policy rate $u_t^* = -\frac{\eta(a_1 + 2a_{11}m + a_{12}v) + \kappa v}{2R_u}$, creating a more negative $u_t^*$ and causing the policy rate to saturate at its lower bound $u_{\min}$.

\subsection{Loss of control}\label{sect:4.4}
We further analyze the saturation of the policy instruments identified in \cref{sect:4.3}. \Cref{figu:loss} presents the loss-of-control diagnostics over the $(\chi,\beta)$ plane, produced according to \Cref{algo:loss}. The base layer is a heatmap of total time at bounds (the share of the horizon during which either control is at its limit), over which we overlay iso-cost contours of $J$. Axes place $\chi$ on the horizontal and $\beta$ on the vertical. Adversary intensities are fixed at $\lambda_m=\lambda_v=0.02$.

\begin{algorithm}[htbp!]
\caption{Loss-of-control map (\Cref{figu:loss})}\label{algo:loss}
\begin{algorithmic}[1]
\State \textbf{Input:} Baseline $p_0$; grids for effectiveness $\chi$ and reversion $\beta$.
\State \textbf{Output:} Heatmap of saturation and iso-cost contours.
\For{each $\chi_i$ in grid, each $\beta_j$ in grid}
    \State Update $p \leftarrow p_0$ with $\chi \leftarrow \chi_i, \beta \leftarrow \beta_j$.
    \State \textbf{Stability Check:} 
    \If{Robustness-breakdown ($4\lambda_v > \frac{\chi^2}{R}$ or $4\lambda_m > \frac{\eta^2}{R_u}$)}
        \State Mark point $(i,j)$ as ``Robustness-breakdown''
    \Else
        \State Simulate and compute metrics
        \State $H_{ij} \leftarrow$ Total time at bounds $(S_u + S_\pi)$
        \State $J_{ij} \leftarrow$ Total cost
    \EndIf
\EndFor
\State Plot heatmap of $H_{ij}$ with overlaid contours of $J_{ij}$
\end{algorithmic}
\end{algorithm}

\begin{figure}[htbp!]\centering\includegraphics[width=0.65\textwidth]{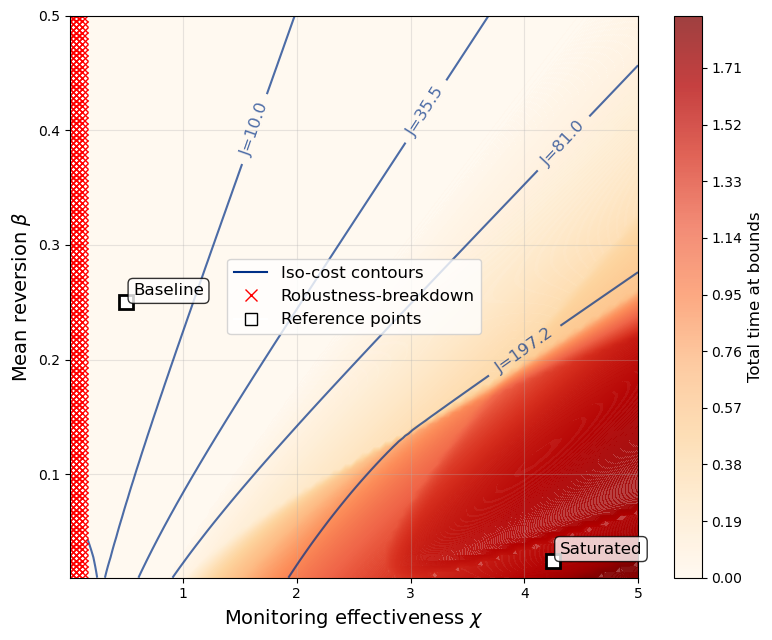}\caption{Loss-of-control map over $(\chi,\beta)$. This figure identifies the two distinct phase transitions: robustness-breakdown and control saturation. It shows total time at bounds, defined as the sum of saturation percentages for both controls ($S_u + S_\pi \in [0,2]$). Iso-cost contours overlay the saturation regions, and red crosses mark robustness-breakdown points. Note that the increasing trends in iso-cost contours in higher $\chi$ regions reflect the over-monitoring cost (Remarks~\ref{rema:over} and~\ref{rema:over_scope}).}\label{figu:loss}
\end{figure}

\Cref{figu:loss} illustrates the central trade-off between the system's inherent stability (provided by $\beta$) and the CB's monitoring effectiveness $\chi$. We identify two distinct phase transitions corresponding to a loss-of-control. First, a robustness-breakdown (the ``Robustness-breakdown'' points, see Remark~\ref{rema:break}) occurs at $\chi \le \sqrt{4\lambda_v R}$ (approximately $0.14$ under baseline $\lambda_v$ and $R$), where the CB's stabilizing control term ($\frac{\chi^2}{R}$) is too weak to offset the destabilizing adversarial term ($4\lambda_v$), preventing the Riccati system from admitting a stable solution. Second, the heatmap exhibits control saturation (the neighborhood of the ``Saturated'' point). The borderline of this saturated region represents the critical policy threshold where controls first reach their bounds. This saturation frontier is crossed when $\beta$ decreases (forcing the CB to compensate for the lack of inherent mean reversion) or when $\chi$ increases. This latter effect is twofold: increasing $\chi$ not only drives $\pi_t^*$ to its bound ($\pi_{max}$) but also causes $u_t$ to saturate at its bound through the Riccati $a_{12}$ coupling. The ``Baseline'' reference point exists in the stable, interior region. Overall, the plot shows a phase transition from interior control to high boundary usage of policy instruments as $\chi$ increases and $\beta$ declines, and the iso-cost contours broadly track these saturation gradients. At the baseline parameters, optimizing the CB's objective $J$ maintains the policy instruments within the interior region, but this property is lost as parameters approach the saturation frontier.

\section{Discussion}\label{sect:discuss}
\paragraph{Simulation results.}
Our simulations show that the total cost $J$ is highly sensitive to the variance adversary $\lambda_v$ but relatively insensitive to the mean adversary $\lambda_m$ within the stable region. This asymmetry is driven by the \emph{net stabilizing margins} in the Riccati dynamics (defined by the difference between control effectiveness and adversarial pressure: $\frac{\eta^2}{R_u} - 4\lambda_m$ versus $\frac{\chi^2}{R} - 4\lambda_v$). Under our baseline parameters, the stabilizing coefficient for the mean ($\frac{\eta^2}{R_u} = 1.28$) exceeds that of the variance ($\frac{\chi^2}{R} = 1.00$), creating a larger safety margin against adversarial distortions. Consequently, the optimal distortion $\xi_t$ induces larger state deviations in $v_t$ compared to the effect of $\theta_t$ on $m_t$. This effect is visible in the path simulations: while the baseline case is affected by over-monitoring at $v_t=0$, a strong adversary overwhelms the control effort, preventing $v_t$ from settling at zero.

The parameter sweeps reveal complex coordination between controls. For example, when the mean adversary $\lambda_m$ increases, the monitoring policy $\pi_t$ decreases as the CB shifts resources to the mean channel via the $w_2(u_t)$ coupling. Notably, when monitoring effectiveness $\chi$ is large, the total cost $J$ can increase with $\chi$ even though the value function $V(0, x_0)$ decreases. This reflects a structural limitation of applying unconstrained LQ feedback to state-constrained dynamics: high $\chi$ drives the variance $v_t$ to its lower bound of zero, where the Riccati feedback continues to prescribe positive monitoring $\pi_t > 0$ based on the positive $\partial_v V$. The resulting over-monitoring accumulates as $R\pi_t^2$ cost without yielding further variance reduction (Remark~\ref{rema:over}). This coordination effect is also observed in the saturation regimes: increasing $\chi$ not only causes $\pi_t$ to reach its bound, but effectively forces $u_t$ to saturate at its bound by strengthening the $a_{12}$ Riccati coupling, demonstrating that the two policy tools are deeply linked via the cross-terms in the value function.

Finally, our loss-of-control analysis illustrates the central trade-off between the system's inherent stability and the CB's monitoring effectiveness. We identify two distinct, structural phase transitions. First, a robustness-breakdown occurs when the control power is insufficient to offset the adversarial strength (\emph{e.g.,} when $\chi$ is low), violating the Riccati stability condition ($4\lambda_v > \frac{\chi^2}{R}$ or $4\lambda_m > \frac{\eta^2}{R_u}$). Second, control saturation emerges when the unconstrained optimal feedback exceeds the admissible bounds, a regime triggered either when high $\chi$ amplifies the feedback gain beyond $\pi_{max}$ or when low $\beta$ necessitates intervention beyond $u_{min}$.

\paragraph{Policy implications.}
An important policy implication is that monetary policy $u_t$ and bank monitoring $\pi_t$ are deeply coordinated and cannot be managed in isolation. Our model reveals critical, structural trade-offs. For instance, a strong monetary policy response to the mean adversary $\lambda_m$ optimally forces the CB to reduce $\pi_t$ as a resource-allocation shift—a direct, parameter-independent consequence of the $w_2(u_t)$ coupling. Furthermore, the robustness-breakdown is determined by the control channel with the minimum stability margin (the smaller of $\frac{\eta^2}{R_u} - 4\lambda_m$ and $\frac{\chi^2}{R} - 4\lambda_v$). Since the Riccati solution explodes if the quadratic coefficient becomes positive for either channel, the model's capacity to absorb uncertainty is limited by the tighter of these two margins. Under our baseline parameters where $\frac{\eta^2}{R_u} > \frac{\chi^2}{R}$, the variance channel has the tighter margin, creating a specific vulnerability to $\lambda_v$ that would only reverse if $\frac{\chi^2}{R}$ were increased significantly. Furthermore, increasing monitoring effectiveness $\chi$ can trigger saturation in other policy tools due to complex Riccati feedback, or lead to inefficient over-monitoring when variance reaches its lower bound, where $\pi^* \propto \partial_v V$ continues to incur quadratic costs $R\pi^2$ despite yielding zero marginal reduction in $v_t$. These findings suggest that CBs must coordinate (jointly optimize) their policy instruments, accounting for the cross-channel feedback that links these tools.

\paragraph{Limitations.}
Our analysis is subject to several limitations inherent to the LQ-MFC framework.
\begin{enumerate}
    \item Linear dynamics and quadratic objectives, while analytically transparent, may overlook nonlinearities that emerge under extreme stress.
    \item The assumption of continuous controls may ignore the discrete or binary nature of certain policy interventions, such as emergency lending facilities or bank resolution decisions.
    \item The LQ cost structure presumes symmetric penalties for deviations above and below targets, whereas CBs may face asymmetric loss functions in which undersupplying liquidity is far more costly than oversupplying it during crises.
\end{enumerate}
These limitations suggest that the phase transitions we identify may occur at lower adversary strengths, implying greater fragility than our baseline model indicates. Furthermore, our unconstrained Riccati solution may not optimally handle the state constraint $v_t \geq 0$: when variance reaches zero, the feedback policy prescribes inefficient over-monitoring (see Remarks~\ref{rema:over} and~\ref{rema:over_scope}).

\paragraph{Potential model extensions.}
Our framework serves as a tractable baseline that admits several extensions.
\begin{enumerate}
\item Incorporating \emph{jump-diffusions} would capture fire-sale shocks, where Poisson jumps model sudden asset liquidations generating discontinuous increases in cross-sectional dispersion.
\item \emph{Exogenous regime-switching} parameters could model transitions between normal and crisis states, where model coefficients depend on a finite Markov chain $k_t \in \{1,\dots,K\}$ independent of the state.
\item Introducing \emph{network-weighted mean exposure} $\sum_j \omega_{ij} L_t^j$ would capture heterogeneous spillovers and core--periphery dynamics.
\end{enumerate}

These extensions would likely reinforce our understanding that the system's true stability region may be narrower than our baseline model suggests. As detailed in our stability analysis, the system fails when the adversary's strength overpowers the control channels. While exogenous regime switching remains tractable (yielding a system of 
coupled Riccati equations), it would generally shrink the stability region, 
as the coupled solution must remain bounded across all regime configurations. Furthermore, jump-diffusions would introduce a strictly positive, uncontrolled source of dispersion ($\lambda_J \mathbb{E}[Z^2]$) that elevates the variance floor, reducing the effective buffer against destabilization. Together, these mechanisms suggest that the stable region identified in our baseline model may underestimate the system's true fragility.

\paragraph{Future research directions.}
When nonlinear monitoring costs or partial observability prevent closed-form Riccati solutions, numerical solutions to the robust MFC problem can be obtained using PG methods.\footnote{This provides a foundation for applying data-driven reinforcement learning to systemic risk management. Recent developments in PG methods include Giegrich, Reisinger, and Zhang~\cite{GiegrichReisingerZhang2024}, Reisinger, Stockinger, and Zhang~\cite{ReisingerStockingerZhang2024}, and Hambly, Xu, and Yang~\cite{HamblyXuYang2021}.} A natural extension is to formulate an MFG where individual banks optimize their liquidity positions while the CB sets aggregate policy. In this setting, the Riccati equations are replaced by a coupled forward-backward system, whose fixed point characterizes the equilibrium between banks and the CB. For recent developments in MFG formulations, see Cont and Hu~\cite{ContHu2025}. Moreover, the single CB framework can be extended to include multiple regulatory authorities, in the spirit of Veraart and Aldasoro~\cite{VeraartAldasoro2025}.\footnote{This aligns with institutional practice in Japan, where both the Financial Services Agency (primary regulator) and the Bank of Japan monitor banks.}

\section{Conclusion}\label{sect:concl}
\paragraph{Main contributions.}
This paper contributes to the systemic risk literature in financial mathematics by developing a robust LQ-MFC framework that incorporates multiple coordinated policy instruments under model uncertainty. We jointly optimize interest rate policy $u_t$ and supervisory monitoring intensity $\pi_t$ against worst-case distortions. The distinguishing feature is a variance weight $w_2(u_t) = \bar{w}_2 + \kappa u_t$ that depends on the policy rate, generating coupling effects between monetary policy and cross-sectional dispersion, creating control--variance interactions in the optimal control Hamiltonian. This coupling captures the heterogeneous transmission of monetary policy through the banking system, where higher policy rates amplify dispersion costs—a mechanism that emerges from bank heterogeneity but is absent when $\kappa = 0$.

\paragraph{Theoretical foundations.}
We establish viscosity solutions for the robust HJBI equation, prove uniqueness through comparison principles, and provide verification theorems for optimal feedback controls. The LQ structure with the coupling term $\kappa$ leads to a modified Riccati system that admits closed-form solutions despite the control--state interactions. These analytical results maintain tractability and enable practical numerical implementation.

\paragraph{Instrument complementarity.}
Our model demonstrates complementarity between interest rate and monitoring policies under model uncertainty. Monitoring primarily affects variance while interest rates target the mean, creating distinct stabilization channels. The policy rate--variance coupling through $\kappa$ generates interaction effects: monetary policy influences cross-sectional dispersion, while heterogeneity feeds back into optimal rate-setting. Under robustness concerns, this coupling becomes particularly important as the adversary optimally targets the channel with lower control effectiveness (the smaller of $\frac{\eta^2}{R_u}$ and $\frac{\chi^2}{R}$).

\paragraph{Coordinated policy responses.}
Our simulations reveal that the system's sensitivity to model uncertainty is driven by the net stabilizing margins. Under our baseline parameters, this results in a pronounced vulnerability to the variance adversary $\lambda_v$ while remaining robust to the mean adversary $\lambda_m$. This asymmetry arises from the higher effectiveness of the interest rate policy relative to monitoring. Beyond this asymmetry, we identify fundamental coordination trade-offs. For instance, an increase in $\lambda_m$ forces a decrease in monitoring $\pi_t$ as the CB reallocates resources via the $w_2(u_t)$ coupling.

We further identify two distinct loss-of-control regimes: robustness-breakdown and control saturation. A breakdown threshold exists where system stability fails when the adversary's influence overwhelms the CB's control effectiveness. Separately, control saturation arises from complex Riccati interactions, even within the stable region. For instance, increasing monitoring effectiveness $\chi$ not only drives $\pi_t$ to its bound but also raises total cost $J$ (due to over-monitoring when $v_t$ reaches zero) and pushes the policy rate $u_t$ to its bound via the $a_{12}$ Riccati coupling.

\paragraph{Possible extensions.}
Important extensions include incorporating jump processes for sudden liquidity shocks and exogenous regime-switching dynamics (see \cref{sect:discuss}). The framework developed here—combining robust control, mean-field approximations, and multiple instruments—offers a foundation for future developments.

\begin{appendices}
\addcontentsline{toc}{section}{Appendices}
\addtocontents{toc}{\protect\setcounter{tocdepth}{0}}

\section{Propagation-of-Chaos for the \emph{N}-bank system}\label{appe:A}
This appendix establishes PoC (Propagation-of-Chaos) for the $N$-bank system, justifying the mean-field approximation used in the main text (see Carmona and Delarue~\cite{CarmonaDelarue2013}). While the main text analyzes deterministic moment dynamics, the underlying bank system is stochastic with both idiosyncratic and common noise.

Consider $N$ banks with liquidity gaps following:
\[dL_t^{i,N} = \Big[ -\beta\big(L_t^{i,N} - m_t^N\big) + \eta\,u_t + \theta_t \Big]\,dt \;+\; \sigma_L\,dW_t^{i} \;+\; \sigma_c\,dB_t, \quad i=1,\dots,N,\]
where $m_t^N = \tfrac{1}{N}\sum_{j=1}^N L_t^{j,N}$ is the empirical mean, $\{W^i\}_{i=1}^N$ are i.i.d. standard Brownian motions, and $B$ is a common Brownian motion independent of all $W^i$.

Assume $(L_0^{i,N})_{i=1,\dots,N}$ are i.i.d. with $\mathbb{E}|L_0^{i,N}|^2 < \infty$ (and $\mathbb{E}|L_0^{i,N}|^4 < \infty$ for Corollary~\ref{coro:poc}), independent of $(B, W^1,\dots,W^N)$. We use the conditional (common-noise) law and mean
\[\mu_t := \mathcal{L}(L_t \mid \mathcal{F}_t^B), \quad m_t := \mathbb{E}[L_t \mid \mathcal{F}_t^B], \quad \mathcal{F}_t^B := \sigma(B_s: s\le t).\]

The controls $(u_t, \pi_t)$ and adversarial distortions $(\theta_t, \xi_t)$ are progressively measurable and bounded processes, chosen according to the robust control problem in \cref{sect:theor}. Note that $\pi_t$ affects variance but not individual dynamics directly, while $\xi_t$ enters only at the aggregate level.

\begin{assumption}[regularity for PoC]\label{assu:F}
The parameters satisfy $\beta, \eta > 0$ and $\sigma_L, \sigma_c \ge 0$. Controls are bounded: $u_t \in [u_{\min}, u_{\max}]$, $\pi_t \in [0, \pi_{\max}]$. Adversarial distortions satisfy $|\theta_t| \leq C_{\theta}$ and $|\xi_t| \leq C_{\xi}$ for constants $C_\theta, C_\xi$ determined by the KL penalty parameter $\lambda$.
\end{assumption}

Assumption~\ref{assu:F} holds under Assumption~\ref{assu:D} on the baseline parameter region considered.
\begin{lemma}[uniform second moments]
Under Assumption~\ref{assu:F}, there exists $C_T < \infty$ such that, for all $N \geq 1$:
\[\sup_{1 \leq i \leq N} \mathbb{E}\Big[\sup_{0 \leq t \leq T} |L_t^{i,N}|^2\Big] \leq C_T\big(1 + \mathbb{E}|L_0^{i,N}|^2\big).\]
\end{lemma}

\begin{proof}
Apply Ito's formula to $|L_t^{i,N}|^2$. The drift satisfies, using boundedness of $u_t,\theta_t$ and $\mathbb{E}[L_t^{i,N} m_t^N] \le (\mathbb{E}|L_t^{i,N}|^2)^{1/2} (\mathbb{E}|m_t^N|^2)^{1/2}$,
\[\mathbb{E}\big[2L_t^{i,N}(-\beta(L_t^{i,N} - m_t^N) + \eta u_t + \theta_t)\big] \le -\beta\,\mathbb{E}|L_t^{i,N}|^2 + C\big(1 + \mathbb{E}|L_t^{i,N}|^2\big).\]
The quadratic variation contributes $\sigma_L^2 + \sigma_c^2$. By the Burkholder--Davis--Gundy inequality (\emph{e.g.,} Karatzas and Shreve~\cite{Karatzas}, Theorem 3.28) and Young's inequality (\emph{e.g.,} Royden and Fitzpatrick~\cite{RoydenFitzpatrick2010}, Section 7.2) applied to the martingale terms, followed by Gronwall's inequality (\emph{e.g.,} Øksendal~\cite{OksendalSDE}, Chapter 5), we obtain the stated sup-in-time bound.
\end{proof}

\begin{theorem}[propagation-of-chaos]
Let $(L_t^i)_{i\ge 1}$ be i.i.d. copies of the McKean--Vlasov limit solving
\[dL_t^i = \big[-\beta(L_t^i - m_t) + \eta u_t + \theta_t\big]\,dt + \sigma_L\,dW_t^i + \sigma_c\,dB_t,\]
where $m_t = \mathbb{E}[L_t^i \mid \mathcal{F}_t^B]$. Under Assumption~\ref{assu:F} and synchronous coupling (same Brownian motions $W^i$ and $B$), there exists $C_T < \infty$ such that
\[\max_{1 \leq i \leq N} \mathbb{E}\Big[\sup_{0 \leq t \leq T} |L_t^{i,N} - L_t^i|^2\Big]^{1/2} \leq C_T N^{-1/2}.\]
\end{theorem}

\begin{proof}
Let $e_t^i := L_t^{i,N} - L_t^i$. By synchronous coupling, the noise cancels and
\[de_t^i = -\beta\big(e_t^i - (m_t^N - m_t)\big)\,dt,\quad \frac{d}{dt}\,\mathbb{E}|e_t^i|^2 \;\le\; C\Big(\mathbb{E}|e_t^i|^2 + \mathbb{E}|m_t^N - m_t|^2\Big).\]
Decompose the mean error with $\bar m_t^N := \tfrac{1}{N}\sum_{j=1}^N L_t^{j}$ (i.i.d. mean-field copies):
\[\mathbb{E}|m_t^N - m_t|^2 \le 2\,\mathbb{E}|m_t^N - \bar m_t^N|^2 + 2\,\mathbb{E}|\bar m_t^N - m_t|^2 \le \frac{2}{N}\sum_{j=1}^N \mathbb{E}|L_t^{j,N} - L_t^{j}|^2 + \frac{2}{N}\,\mathbb{E}\big[\mathrm{Var}(L_t\mid \mathcal{F}_t^B)\big].\]
Plugging this into the differential inequality and applying Gronwall's inequality yields $\mathbb{E}|e_t^i|^2 \le C_T N^{-1}$, hence the stated $N^{-1/2}$ rate for the sup-in-time error by standard maximal inequalities.
\end{proof}

\begin{corollary}[convergence of moments]\label{coro:poc}
Let $v_t^N = \tfrac{1}{N}\sum_{i=1}^N |L_t^{i,N} - m_t^N|^2$ be the empirical variance and assume $\mathbb{E}|L_0^{i,N}|^4 < \infty$. Then
\[\mathbb{E}[|m_t^N - m_t|] \leq C_T N^{-1/2}, \qquad \mathbb{E}[|v_t^N - v_t|] \leq C_T N^{-1/2},\]
where $v_t = \mathrm{Var}(L_t^i \mid \mathcal{F}_t^B)$ is the conditional variance.
\end{corollary}

\paragraph{Convergence rates and finite-sample behavior.}
We briefly discuss convergence rates for practical implementation.

\begin{proposition}[rate of convergence]\label{prop:convergence_rate}
Let $(m^N_t, v^N_t)$ denote the empirical mean and variance of the $N$-bank system, and $(m_t, v_t)$ the mean-field limits. Under Assumption~\ref{assu:F}, for any $T > 0$,
\begin{equation*}
\mathbb{E}\left[\sup_{t \in [0,T]} |m^N_t - m_t|^2 + |v^N_t - v_t|^2\right] \leq \frac{C(T)}{N},
\end{equation*}
where $C(T)$ depends on the system parameters and grows polynomially in $T$.
\end{proposition}

The $O(1/\sqrt{N})$ convergence rate in $L^2$ is standard for mean-field limits with Lipschitz coefficients. The mean-field limit becomes increasingly accurate for larger banking systems.

\begin{remark}[implications]
The dependence of the convergence constant $C(T)$ on terminal cost weights $G_m, G_v$ suggests that stronger terminal penalties require larger $N$ for accurate approximation. However, the robustness-breakdown thresholds (Remark~\ref{rema:break}) are determined by the Riccati coefficients and remain valid independent of $N$.
\end{remark}

\section{Technical proofs}\label{appe:B}
\subsection{Proof of \Cref{theo:visc}}\label{proo:visc}
\begin{proof}
We adapt the standard viscosity argument (\emph{e.g.,} \O{}ksendal and Sulem~\cite{OS-SCJD}, Chapter 12) to absolutely continuous dynamics. Let $(t_0,x_0)$ and $\varphi\in C^1$ be such that $V-\varphi$ attains a local minimum at $(t_0,x_0)$ with $V(t_0,x_0)=\varphi(t_0,x_0)$. For small $\delta>0$ and any stopping time $\tau_\delta\in[t_0,t_0+\delta]$, the DPP yields
\[V(t_0,x_0) \le \inf_{u,\pi} \sup_{\theta,\xi}\Big[ \int_{t_0}^{\tau_\delta} \ell(X_s,u_s,\pi_s,\theta_s,\xi_s)\,ds + V(\tau_\delta,X_{\tau_\delta}) \Big].\]
Since $V\ge \varphi$ and they touch at $(t_0,x_0)$,
\[0 \le \inf_{u,\pi} \sup_{\theta,\xi}\Big[ \int_{t_0}^{\tau_\delta} \ell(\cdot)\,ds + \varphi(\tau_\delta,X_{\tau_\delta}) - \varphi(t_0,x_0) \Big].\]
As $X$ is absolutely continuous and $\varphi\in C^1$,
\[\varphi(\tau_\delta,X_{\tau_\delta}) - \varphi(t_0,x_0) = \int_{t_0}^{\tau_\delta} \Big( \partial_t \varphi + \nabla_x \varphi \cdot \dot X_s \Big)\, ds + o(\delta),\]
with $o(\delta)/\delta \to 0$ uniformly by compactness of action sets. Dividing by $\delta$ and letting $\delta\downarrow 0$ yields
\[0 \ge \inf_{u,\pi}\sup_{\theta,\xi} \Big\{ \ell(x_0,u,\pi,\theta,\xi) + \partial_t \varphi(t_0,x_0) + \nabla_x \varphi(t_0,x_0)\cdot b(x_0,u,\pi,\theta,\xi) \Big\},\]
with $b(x,u,\pi,\theta,\xi)=(\eta u+\theta,\,-2\beta v + \sigma_L^2+\sigma_c^2+\xi-\chi\pi)$ as defined in \cref{sec:model_primitives}. Optimizing in $(\theta,\xi)$ gives the Hamiltonian representation. Hence
\[- \partial_t \varphi(t_0,x_0) + H\big(x_0,\nabla_x \varphi(t_0,x_0)\big) \ge 0,\]
which proves Part 1. The terminal condition and Part 2 follow by the standard viscosity tests (\emph{e.g.,} Fleming and Soner~\cite{FlemingSoner2006}) with functions touching from above/below at $t=T$. Part 3 follows from Isaacs’ condition, which is ensured under Assumption~\ref{assu:A}.
\end{proof}

\subsection{Proof of \Cref{theo:comparison}}\label{proo:comparison}
\begin{proof}
We use doubling-of-variables and penalization (Crandall, Ishii, and Lions~\cite{CrandallIshiiLions1992} and Fleming and Soner~\cite{FlemingSoner2006}), within the constrained viscosity framework on $\mathbb{R}\times\mathbb{R}_+$.

\paragraph{Coercive penalization and sup of gaps.}
Define, for $\alpha,\gamma>0$,
\[\Psi_{\alpha,\gamma}(t,s,x,y):= U(t,x) - W(s,y) - \frac{|x-y|^2}{2\alpha} - \frac{|t-s|^2}{2\gamma}- \zeta \big(1+|x|^{2q}+|y|^{2q}\big),\]
with small parameters $\zeta>0$ and integer $q\ge 1$ dominating the polynomial growth. Let $(t_{\alpha},s_{\alpha},x_{\alpha},y_{\alpha})$ maximize $\Psi_{\alpha,\gamma}$ on $[0,T]\times[0,T]\times(\mathbb{R}\times\mathbb{R}_+)^2$. By coercivity and growth control, $|x_{\alpha}-y_{\alpha}|\to 0$ and $|t_{\alpha}-s_{\alpha}|\to 0$ as $\alpha,\gamma\downarrow 0$. Moreover, $\sup \Psi_{\alpha,\gamma} \to \sup_{t,x} (U-W)$ as $\alpha,\gamma\downarrow 0$ then $\zeta\downarrow 0$.

\paragraph{Ishii’s lemma.}
There exist jets
\[(a_X,p_X,X)\in \overline{\mathcal{P}}^{2,+} U(t_{\alpha},x_{\alpha}),\quad (a_Y,p_Y,Y)\in \overline{\mathcal{P}}^{2,-} W(s_{\alpha},y_{\alpha}),\]
with
\[a_X = \tfrac{t_{\alpha}-s_{\alpha}}{\gamma},\quad a_Y = -\tfrac{t_{\alpha}-s_{\alpha}}{\gamma},\]
\[p_X = \tfrac{x_{\alpha}-y_{\alpha}}{\alpha} - \zeta\,\nabla_x\big(|x|^{2q}\big)\big|_{x=x_{\alpha}},\quad p_Y = \tfrac{x_{\alpha}-y_{\alpha}}{\alpha} + zeta\,\nabla_y\big(|y|^{2q}\big)\big|_{y=y_{\alpha}}.\]
The second-order matrices are coupled in the standard way by Ishii’s lemma. Since the HJBI is first order, only the first-order components enter the inequalities below. The additional $\zeta$-gradient terms vanish as $\zeta\downarrow 0$ uniformly on bounded sets.

\paragraph{Sub-/super- inequalities.}
By the viscosity properties (with constrained semijets),
\[-a_X + H(x_{\alpha},p_X) \le 0, \quad -a_Y + H(y_{\alpha},p_Y) \ge 0.\]
Subtracting the second inequality from the first yields
\[H(x_{\alpha},p_X) - H(y_{\alpha},p_Y) \le a_X - a_Y = 2\, \tfrac{t_{\alpha}-s_{\alpha}}{\gamma}.\]
Hence, taking $\limsup$ as $\alpha,\gamma \downarrow 0$ and using continuity of $H$ together with $|t_{\alpha}-s_{\alpha}| \to 0$, $|x_{\alpha}-y_{\alpha}|\to 0$, and the fact that the $\zeta$-terms vanish as $\zeta\downarrow 0$,
\[\limsup_{\alpha,\gamma\downarrow 0} \big[ H(x_{\alpha},p_X) - H(y_{\alpha},p_Y) \big] \le 0.\]

\paragraph{Continuity and stability of $H$.}
By Assumption~\ref{assu:B}(2) and $|x_{\alpha}-y_{\alpha}|\to 0$, $|p_X-p_Y|\to 0$ (including the $\zeta$-gradient contributions),
\[\limsup_{\alpha,\gamma\downarrow 0} \big[ H(y_{\alpha},p_Y) - H(x_{\alpha},p_X) \big] \le 0.\]
Taking lim sup in the previous inequality yields zero, and the penalization construction implies $\sup_{t,x} \big( U(t,x) - W(t,x) \big) \le 0,$ hence $U\le W$ on the whole domain.

\paragraph{Terminal time and boundary.}
The time penalization and the viscosity attainment of $U(T,\cdot)\le g \le W(T,\cdot)$ preclude a positive gap at $t=T$. At $v=0$, the constrained-viscosity framework on the closed set $\mathbb{R}\times\mathbb{R}_+$ (Assumption~\ref{assu:B}(4)) applies. No explicit boundary condition is imposed, and test functions are taken from the interior.
\end{proof}

\subsection{Proof of \Cref{theo:existence}}\label{proo:existence}
\begin{proof}
We construct time-discrete approximations $V^{\Delta}$ using a monotone, stable, and consistent semi-Lagrangian scheme with piecewise-constant controls, enforcing viability at $v \ge 0$ (no step leaves the domain). Set $V^{\Delta}(T,\cdot)=g$ and, for $t_n = T - n\Delta$,
\begin{equation}
\begin{aligned}
V^{\Delta}(t_n,x)
= \inf_{(u,\pi)\in \mathcal U \times \mathcal P} \; \sup_{(\theta,\xi)\in \Theta \times \Xi}
\Big\{ \Delta\, \tilde{\ell}(x,u,\pi,\theta,\xi) + V^{\Delta}\big(t_{n+1},\, x + \Delta\, b(x,u,\pi,\theta,\xi)\big) \Big\},
\end{aligned}
\end{equation}
with the step restricted to $x' = x + \Delta\, b(\cdot)$ satisfying $v' \ge 0$. The KL penalty is absorbed in the running cost $\tilde{\ell}$, which is coercive in $(\theta,\xi)$. The scheme is monotone. Stability follows from either boundedness of controls (compact case) or the coercive lower bound in $(u,\pi)$ (unbounded case), together with at-most-linear growth of $b$, and it is consistent with the HJBI.

By the Barles--Souganidis framework~\cite{BarlesSouganidis1991}, monotone, stable, and consistent schemes for equations with a comparison principle converge to the unique viscosity solution. Using half-relaxed limits, $V^{\Delta}$ converges locally uniformly to $V$. Polynomial growth follows from discrete Gronwall bounds for the Euler step and, in the unbounded case, the coercive quadratic terms in $(u,\pi)$.

Finally, if Isaacs' condition holds (so $H$ is the common Isaacs Hamiltonian), the robust value function defined via the DPP is a viscosity solution of the same HJBI with the same growth. By comparison, it coincides with $V$.
\end{proof}

\subsection{Proof of \Cref{theo:veri}}\label{proo:veri}
\begin{proof}
By the DPP in Proposition~\ref{prop:DPP} and \Cref{theo:comparison}, any viscosity solution with the prescribed growth and terminal condition must coincide with the robust value function. By Berge's maximum theorem (Berge~\cite{Berge1997}, Section VI-3), existence of measurable minimizers follows from compact action sets and continuity of the Hamiltonian in the controls. The standard super-/submartingale verification argument (\emph{e.g.,} Pham~\cite{Pham2009}, Section 3.5) applied to $t\mapsto \Phi(t,X_t)$ along admissible trajectories yields optimality of the minimizing feedback controls. The state-constraint boundary at $v=0$ is handled in the viscosity sense as in \cref{sect:comparison}.
\end{proof}

\section{Sensitivity analysis and comparative statics}\label{appe:C}
In this appendix, we consider the Riccati sensitivity ODE and Lipschitz comparative statics for the value function. We also derive bounds on value function losses under drift misspecification.

\subsection{Setup and Riccati sensitivity ODE}\label{sect:setupRiccati}
Let $a(t)$ collect the six coefficients of the quadratic value function from \cref{eq:ansatz}, evolving under the interior Riccati map $F$ in the Riccati ODE system \cref{eq:Riccati} with terminal vector $a_T$ in \cref{eq:app_terminal}. Denote parameters by the vector $\Theta$. We restrict our analysis to a compact parameter set $\mathcal{K}$.

We invoke the following local differentiability regularity and interior no-switching assumptions.
\begin{assumption}[local differentiability regularity and interior no-switching]\label{assu:D}
We assume:
\begin{enumerate}
\item The primitives $(b,\ell,g)$ and model coefficients (\emph{e.g.}, $\eta,\kappa,R_u,R,\lambda_m,\lambda_v,\beta,\Sigma^2$) are $C^1$ in $\Theta$ on compact subsets of the parameter space, with locally bounded partial derivatives. The terminal map $a_T(\Theta)$ is $C^1$ in $\Theta$. On interior regions, the right-hand side $F$ in the Riccati ODE system \cref{eq:Riccati} is $C^1$ in $(a,\Theta)$ and locally Lipschitz in $a$, uniformly on compacts.
\item On the time interval under consideration, the projections in \cref{eq:selectors} are inactive. There exist $\delta>0$ and $\varepsilon>0$ such that, for the baseline parameter $\Theta$ and any $\Theta'$ with $\|\Theta'-\Theta\|\le \varepsilon$ (where $\Theta' \in \mathcal K$), the selector arguments on $[t,T]$ remain at least $\delta$ away from the projection boundaries, uniformly along the closed-loop trajectories considered.
\end{enumerate}
\end{assumption}

By the interior no-switching assumption (Assumption~\ref{assu:D}(2)), we remain in the interior (unconstrained) regime of the control selectors. The unconstrained optimizer stays strictly within the admissible sets. Therefore, the projections in \cref{eq:selectors} are inactive, and no threshold is crossed that would change the selector’s formula. Consequently, the selector law is smooth and time-continuous on the interval, with no regime switches or discontinuous shifts in the control law.

Fix a baseline $\Theta$ and a perturbation $\Theta'=\Theta+\delta\Theta$. Under Assumptions~\ref{assu:A} and~\ref{assu:D} and away from selector switching surfaces, the Riccati flow is differentiable in $\Theta$.

\begin{lemma}[Riccati sensitivity ODE]\label{lemm:sensitivity}
Let $D_aF$ and $D_\Theta F$ denote the Jacobians of the interior Riccati map $F$ evaluated along the baseline trajectory $t\mapsto a(t;\Theta)$. Then the directional derivative $\Delta a(t):=\partial_\Theta a(t;\Theta)[\delta\Theta]$ satisfies
\[\frac{d}{dt}\Delta a(t) = D_aF\big(t,a(t;\Theta),\Theta\big)\,\Delta a(t) + D_\Theta F\big(t,a(t;\Theta),\Theta\big)[\delta\Theta], \quad \Delta a(T)=\partial_\Theta a_T(\Theta)[\delta\Theta],\]
with variation-of-constants representation
\[\Delta a(t)=\Phi(t,T)\,\Delta a(T) +\int_t^T \Phi(t,s)\,D_\Theta F\big(s,a(s;\Theta),\Theta\big)[\delta\Theta]\;ds,\]
where $\Phi(t,s)$ is the principal solution of $\dot\Phi(t,s)=D_aF\big(t,a(t;\Theta),\Theta\big)\,\Phi(t,s)$ with $\Phi(s,s)=I$. If $a_T$ is $\Theta$-independent (\emph{e.g.}, fixed $G_m$ in $a_{11}(T)=G_m$), then $\partial_\Theta a_T(\Theta)[\delta\Theta]=0$.
\end{lemma}

\begin{proof}
On any interval where Assumption~\ref{assu:D} holds, $a(\cdot;\Theta)$ solves the terminal-value ODE $\dot a(t)=F\big(t,a(t;\Theta),\Theta\big)$ with $a(T)=a_T(\Theta)$, where $F$ is $C^1$ in $(a,\Theta)$ and locally Lipschitz in $a$. By $C^1$-dependence of ODE solutions on parameters (\emph{e.g.}, Hartman~\cite{Hartman2002}, Chapter V), $\Theta\mapsto a(\cdot;\Theta)$ is $C^1$ on such intervals. Differentiating yields the stated linear variational equation and the variation-of-constants formula, with $\Phi$ the principal solution. If $a_T$ does not depend on $\Theta$, then $\Delta a(T)=0$. The conclusions hold piecewise between selector switching times. Across switches, use one-sided or generalized derivatives.
\end{proof}

\begin{remark}[computational aspects]
The sensitivity ODE in Lemma~\ref{lemm:sensitivity} can be solved numerically alongside the baseline Riccati system, providing gradient information for optimization or robustness analysis without requiring finite differences.
\end{remark}

For $x=(m,v)$, $\Delta V(t,x)=\sum_i \partial_{a_i}V(t,x)\,\Delta a_i(t)$. In the interior, $\Delta\phi$ follows by differentiating the selectors in \cref{eq:selectors} along $a(\cdot;\Theta)$; with active projections, use piecewise derivatives away from switching times and one-sided derivatives at switch times.

\subsection{Lipschitz comparative statics for the value function}\label{sect:lipschitz}
\begin{theorem}[comparative statics: Lipschitz continuity of value]
Assume Assumptions~\ref{assu:A} and~\ref{assu:D} on an interval where selectors are inactive. Let $\mathcal K$ be a compact set of parameters. Suppose there exists a compact set $\mathcal A\subset\mathbb R^6$ such that, for every $\Theta\in\mathcal K$, the Riccati trajectory $t\mapsto a(t;\Theta)$ remains in $\mathcal A$ for all $t\in[0,T]$. Assume further that on $[0,T]\times \mathcal A\times \mathcal K$ the maps $D_aF$ and $D_\Theta F$ are bounded, and that $a_T$ is $C^1$ with $\sup_{\Theta\in\mathcal K}\|\partial_\Theta a_T(\Theta)\|<\infty$. Then, for any $x$ and $t\in[0,T]$, there exists $C_{T,\mathcal K}>0$ such that for all $\Theta,\Theta'\in\mathcal K$,
\[|V(t,x;\Theta') - V(t,x;\Theta)| \le C_{T,\mathcal K}\, (1+|x|^2)\, \|\Theta'-\Theta\|.\]
If $a_T$ is $\Theta$-independent, $C_{T,\mathcal K}$ can be chosen without the terminal sensitivity term.
\end{theorem}

\begin{proof}
Consider the line $\Theta_s=\Theta+s(\Theta'-\Theta)$, $s\in[0,1]$, and set $\delta\Theta=\Theta'-\Theta$. By Lemma~\ref{lemm:sensitivity}, for each $s$ the directional derivative $\Delta a_s(t):=\partial_\Theta a(t;\Theta_s)[\delta\Theta]$ solves
\[\frac{d}{dt}\Delta a_s(t)=A_s(t)\Delta a_s(t)+b_s(t),\quad \Delta a_s(T)=\partial_\Theta a_T(\Theta_s)[\delta\Theta],\]
with $A_s(t)=D_aF(t,a(t;\Theta_s),\Theta_s)$ and $b_s(t)=D_\Theta F(t,a(t;\Theta_s),\Theta_s)[\delta\Theta]$. By the boundedness assumptions on $[0,T]\times\mathcal A\times\mathcal K$, Gronwall's inequality yields
\[\begin{aligned}\|\Delta a_s(t)\| &\le \Big(\|\partial_\Theta a_T(\Theta_s)\|+\int_t^T \|D_\Theta F(r,\cdot,\cdot)\|\,dr\Big)\\&\phantom{\le}\times \exp\!\Big(\int_t^T \|D_aF(r,\cdot,\cdot)\|\,dr\Big)\,\|\delta\Theta\|\le C_T\,\|\delta\Theta\|,\end{aligned}\]
with $C_T$ uniform in $s\in[0,1]$ and $\Theta\in\mathcal K$. Since $V(t,x;\Theta)$ is quadratic in $x$ with coefficients $a(t;\Theta)$, its $s$-derivative satisfies
\[\Big|\frac{d}{ds}V(t,x;\Theta_s)\Big|\le C(1+|x|^2)\,\|\Delta a_s(t)\|\le C'(1+|x|^2)\,\|\delta\Theta\|.\]
Integrating in $s\in[0,1]$ yields the stated Lipschitz bound with $C_{T,\mathcal K}=C'$. If $a_T$ is independent of $\Theta$, the term involving $\|\partial_\Theta a_T\|$ drops from $C_T$.
\end{proof}

\subsection{Robustness loss bounds under drift misspecification}\label{sect:robustness}
Let $\Theta=(\Theta_d,\Theta_o)$, where $\Theta_d$ collects drift parameters that may be misspecified. Suppose the implemented controller is designed for $\Theta$ but the true model is $\Theta'=(\Theta_d+\delta_d,\Theta_o)$ with $\left\lVert{\delta_d}\right\rVert\le \varepsilon$.

\begin{assumption}[stability under misspecification]
Under the misspecified parameters $\Theta' = (\Theta_d + \delta_d, \Theta_o)$ with $\|\delta_d\| \leq \varepsilon$, the closed-loop system remains stable and the state trajectory $(m_t, v_t)$ remains in the domain $\mathbb{R} \times [0, v_{\max}]$ for some $v_{\max} < \infty$.
\end{assumption}

\begin{theorem}[performance gap bound]
Let $J(t,x;\phi;\Theta')$ denote the realized cost when applying the feedback $\phi(\cdot;\Theta)$ in the true model $\Theta'$. Assume Assumptions~\ref{assu:A},~\ref{assu:B},~\ref{assu:C}, and~\ref{assu:D} hold on a compact parameter set $\mathcal K$ containing both $\Theta$ and $\Theta'$, selectors are inactive on $[t,T]$, the Hamiltonian is uniformly strongly convex in controls (and concave in adversarial terms, if present) with modulus $\mu>0$, and the model coefficients are uniformly Lipschitz in $\Theta$ on $[t,T]\times\mathcal A\times\mathcal K$.

Then, there exists $C_{T,\mathcal K}>0$ such that
\[0 \le J\big(t,x;\phi(\cdot;\Theta);\Theta'\big) - V(t,x;\Theta') \le C_{T,\mathcal K} \, (1+\lvert x\rvert^2)\, \varepsilon^2.\]
In particular, the first-order loss vanishes, and the robustness loss is quadratic in the magnitude of drift misspecification.
\end{theorem}

\begin{proof}
Plug $V(\cdot;\Theta)$ into the true HJB/HJBI at $\Theta'$. Since $V$ solves $-\partial_t V + H_{\Theta}(x,\nabla V) = 0$ and the Hamiltonian is Lipschitz in parameters on $[t,T]\times\mathcal A\times\mathcal K$, the pointwise residual
\[r(t,x) := -\partial_t V(t,x;\Theta) + H_{\Theta'}\big(x,\nabla V(t,x;\Theta)\big) = H_{\Theta'}\big(x,\nabla V(t,x;\Theta)\big) - H_{\Theta}\big(x,\nabla V(t,x;\Theta)\big)\]
is uniformly $\mathcal O(\varepsilon)$ on $[t,T]\times\mathcal A$. Moreover, under the interior no-switching assumption (Assumption~\ref{assu:D}(2)), selectors are smooth, so the Lipschitz bound carries through the feedback map without kinks. Uniform strong convexity (in controls) and concavity (in adversarial terms, if present) with modulus $\mu>0$ imply the standard Hamiltonian error-to-policy suboptimality inequality. When a smooth $W$ is used to construct the feedback by minimizing the $\Theta'$-Hamiltonian, the running suboptimality is controlled by the square of the Hamiltonian residual divided by $\mu$ (this follows by completing the square around the $\Theta'$-optimal control). Applying this with $W = V(\cdot;\Theta)$ yields a per-time integrand bounded by $C\, r(t,X_t)^2/\mu$. Since $r = \mathcal O(\varepsilon)$ uniformly, we obtain an $\mathcal O(\varepsilon^2)$ bound on the instantaneous gap.

To pass from integrands to total cost, evaluate along the closed-loop state $X$ driven by the implemented feedback $\phi(\cdot;\Theta)$ under the true model $\Theta'$. Under the standing linear-growth/Lipschitz assumptions (Assumptions~\ref{assu:A},~\ref{assu:B},~\ref{assu:C}, and~\ref{assu:D}), second moments of $X$ are bounded on $[t,T]$ by Gronwall's inequality, yielding $\mathbb E[1+\lvert X_t\rvert^2] \le C\,(1+\lvert x\rvert^2)$. Combining the strong-convexity estimate with $r = \mathcal O(\varepsilon)$, integrating over $[t,T]$, and using the moment bound yields $0 \le J\big(t,x;\phi(\cdot;\Theta);\Theta'\big) - V(t,x;\Theta') \le C_{T,\mathcal K}\, (1+\lvert x\rvert^2)\, \varepsilon^2,$ for a constant $C_{T,\mathcal K}$ depending on $\mu$, Lipschitz and growth constants, and $T$. Intuitively, the first-order term cancels by the envelope principle under interiority (no switching), so the leading error is quadratic in the drift misspecification. In the LQ case, the same $\varepsilon^2$ rate follows directly by completing the square in the closed-loop cost and bounding the perturbation terms.
\end{proof}

\begin{remark}[differentiability and higher-order terms]
If $F$ is $C^2$ and trajectories remain in a compact set, one can expand $a(t;\Theta')$ to second order and refine constants. In time-homogeneous LQ, explicit Riccati solutions yield closed-form sensitivity matrices and sharper constants.
\end{remark}

\section{Derivation of the Riccati ODE system}\label{appe:D}
This appendix provides the full derivation for the Riccati ODE system \cref{eq:Riccati}.

\paragraph{The optimized Hamiltonian.}
The derivation begins with the HJBI equation $-\partial_t V + H(x, \nabla V) = 0.$ As shown in \cref{sect:HJBI}, we can solve the $\inf$-$\sup$ problem for the controls and distortions analytically. The Isaacs Hamiltonian $H(x,p)$ where $x=(m,v)$ and $p=(p_m, p_v) = (\partial_m V, \partial_v V)$ is:
\begin{align*}
H(x,p) = \inf_{u,\pi} \sup_{\theta,\xi} \Big\{ & p_m(\eta u+\theta) + p_v\big(-2\beta v + \Sigma^2+\xi-\chi\pi\big) \\
& + w_1 m^2 + (\bar{w}_2 + \kappa u) v + R\pi^2 + R_u u^2 - \tfrac{\theta^2}{4\lambda_m} - \tfrac{\xi^2}{4\lambda_v} \Big\}.
\end{align*}
Solving the unconstrained optimization problems (by completing the square or first-order conditions) yields the optimized Hamiltonian $H^*$:
\begin{align*}
H^*(x, p) = & \lambda_m p_m^2 + \lambda_v p_v^2 + w_1 m^2 + \bar{w}_2 v + p_v(-2\beta v + \Sigma^2) - \frac{(\eta p_m + \kappa v)^2}{4R_u} - \frac{\chi^2 p_v^2}{4R}.
\end{align*}

\paragraph{The ansatz and gradients.}
We use the quadratic ansatz from \cref{eq:ansatz}: $V(t,m,v) = a_0 + a_1 m + a_2 v + a_{11} m^2 + a_{12} mv + a_{22} v^2,$ where $a_i = a_i(t)$. The required derivatives are:
\begin{align*}
\partial_t V &= \dot{a}_0 + \dot{a}_1 m + \dot{a}_2 v + \dot{a}_{11} m^2 + \dot{a}_{12} mv + \dot{a}_{22} v^2, \\
p_m = \partial_m V &= a_1 + 2a_{11} m + a_{12} v, \\
p_v = \partial_v V &= a_2 + a_{12} m + 2a_{22} v.
\end{align*}

\paragraph{Matching coefficients.}
We substitute the ansatz and its gradients into the HJBI equation $-\partial_t V + H^*(x, \nabla V) = 0$. This creates a large polynomial in $m$ and $v$. For the equation to hold for all $(m,v)$, the coefficients of each monomial ($m^2, v^2, mv, m, v, 1$) must be equal. We equate the coefficients from $-\partial_t V$ (\emph{e.g.,} $-\dot{a}_{11}$) with the corresponding coefficients from $H^*$. This is equivalent to setting $\dot{a}_i$ equal to the collected coefficients from $H^*$.

\paragraph{Coefficient of constant terms (yields $\dot{a}_{0}$).} For the terms from $H^*$ containing constants, we have: $p_v(\Sigma^2) \implies (a_2)\Sigma^2 = \Sigma^2 a_2$, \quad $\lambda_m p_m^2 \implies \lambda_m (a_1)^2 \implies \lambda_m a_1^2$ $\lambda_v p_v^2 \implies \lambda_v (a_2)^2 \implies \lambda_v a_2^2$, \quad $-\tfrac{(\eta p_m)^2}{4R_u} \implies -\tfrac{1}{4R_u} (\eta a_1)^2 = -\tfrac{\eta^2}{4R_u} a_1^2$, and $-\tfrac{\chi^2 p_v^2}{4R} \implies -\tfrac{\chi^2}{4R} (a_2)^2 = -\tfrac{\chi^2}{4R} a_2^2$.

Therefore, $\dot{a}_{0} = \Sigma^2 a_2 + \Big(\lambda_m - \tfrac{\eta^2}{4R_u}\Big) a_1^2 + \Big(\lambda_v - \tfrac{\chi^2}{4R}\Big) a_2^2$.

\paragraph{Coefficient of $m$ (yields $\dot{a}_{1}$).} For the terms from $H^*$ containing $m$, we have: $p_v(\Sigma^2) \implies (a_{12} m)\Sigma^2 \implies \Sigma^2 a_{12}$, \quad $\lambda_m p_m^2 \implies \lambda_m \cdot 2 \cdot (a_1)(2a_{11} m) \implies 4\lambda_m a_1 a_{11}$, \quad $\lambda_v p_v^2 \implies \lambda_v \cdot 2 \cdot (a_2)(a_{12} m) \implies 2\lambda_v a_2 a_{12}$, \quad $-\tfrac{(\eta p_m)^2}{4R_u} \implies -\tfrac{1}{4R_u} \cdot 2 \cdot (\eta a_1)(\eta 2a_{11} m) \implies -\tfrac{\eta^2}{R_u} a_1 a_{11}$, and $-\tfrac{\chi^2 p_v^2}{4R} \implies -\tfrac{\chi^2}{4R} \cdot 2 \cdot (a_2)(a_{12} m) \implies -\tfrac{\chi^2}{2R} a_2 a_{12}$.

Therefore, $\dot{a}_{1} = \Sigma^2 a_{12} + \Big(4\lambda_m - \tfrac{\eta^2}{R_u}\Big) a_1 a_{11} + \Big(2\lambda_v - \tfrac{\chi^2}{2R}\Big) a_2 a_{12}$.

\paragraph{Coefficient of $v$ (yields $\dot{a}_{2}$).} For the terms from $H^*$ containing $v$, we have: $\bar{w}_2 v \implies \bar{w}_2$, \quad $p_v(-2\beta v) \implies (a_2)(-2\beta v) \implies -2\beta a_2$, \quad $p_v(\Sigma^2) \implies (2a_{22} v)\Sigma^2 \implies 2\Sigma^2 a_{22}$, \quad $\lambda_m p_m^2 \implies \lambda_m \cdot 2 \cdot (a_1)(a_{12} v) \implies 2\lambda_m a_1 a_{12}$, \quad $\lambda_v p_v^2 \implies \lambda_v \cdot 2 \cdot (a_2)(2a_{22} v) \implies 4\lambda_v a_2 a_{22}$, \quad $-\tfrac{(\eta p_m + \kappa v)^2}{4R_u} \implies -\tfrac{1}{4R_u} [ 2(\eta a_1)(\eta a_{12} v + \kappa v) ] \implies -\tfrac{\eta^2}{2R_u} a_1 a_{12} - \tfrac{\eta\kappa}{2R_u} a_1$, and $-\tfrac{\chi^2 p_v^2}{4R} \implies -\tfrac{\chi^2}{4R} \cdot 2 \cdot (a_2)(2a_{22} v) \implies -\tfrac{\chi^2}{R} a_2 a_{22}$.

Therefore, $\dot{a}_{2} = \bar{w}_2 - 2\beta a_2 + 2\Sigma^2 a_{22} + \Big(2\lambda_m - \tfrac{\eta^2}{2R_u}\Big) a_1 a_{12} + \Big(4\lambda_v - \tfrac{\chi^2}{R}\Big) a_2 a_{22} - \tfrac{\eta\kappa}{2R_u} a_1$.

\paragraph{Coefficient of $m^2$ (yields $\dot{a}_{11}$).} For the terms from $H^*$ containing $m^2$, we have: $w_1 m^2 \implies w_1$, \quad $\lambda_m p_m^2 \implies \lambda_m (2a_{11} m)^2 \implies 4\lambda_m a_{11}^2$, \quad $\lambda_v p_v^2 \implies \lambda_v (a_{12} m)^2 \implies \lambda_v a_{12}^2$, \quad $-\tfrac{(\eta p_m)^2}{4R_u} \implies -\tfrac{1}{4R_u} (\eta (2a_{11} m))^2 \implies -\tfrac{\eta^2}{R_u} a_{11}^2$, and $-\tfrac{\chi^2 p_v^2}{4R} \implies -\tfrac{\chi^2}{4R} (a_{12} m)^2 \implies -\tfrac{\chi^2}{4R} a_{12}^2$.

Therefore, $\dot{a}_{11} = w_1 + \Big(4\lambda_m - \tfrac{\eta^2}{R_u}\Big) a_{11}^2 + \Big(\lambda_v - \tfrac{\chi^2}{4R}\Big) a_{12}^2$.

\paragraph{Coefficient of $mv$ (yields $\dot{a}_{12}$).} For the terms from $H^*$ containing $mv$, we have: $p_v(-2\beta v) \implies (a_{12}m)(-2\beta v) \implies -2\beta a_{12}$, \quad $\lambda_m p_m^2 \implies \lambda_m \cdot 2 \cdot (2a_{11} m)(a_{12} v) \implies 4\lambda_m a_{11} a_{12}$, \quad $\lambda_v p_v^2 \implies \lambda_v \cdot 2 \cdot (a_{12} m)(2a_{22} v) \implies 4\lambda_v a_{12} a_{22}$, \quad $-\tfrac{(\eta p_m + \kappa v)^2}{4R_u} \implies -\tfrac{1}{4R_u} \cdot 2 \cdot (\eta 2a_{11} m)(\eta a_{12} v + \kappa v) \implies -\tfrac{\eta^2}{R_u} a_{11} a_{12} - \tfrac{\eta\kappa}{R_u} a_{11}$, and $-\tfrac{\chi^2 p_v^2}{4R} \implies -\tfrac{\chi^2}{4R} \cdot 2 \cdot (a_{12} m)(2a_{22} v) \implies -\tfrac{\chi^2}{R} a_{12} a_{22}$.

Therefore, $\dot{a}_{12} = -2\beta a_{12} - \tfrac{\eta\kappa}{R_u} a_{11} + \Big(4\lambda_m - \tfrac{\eta^2}{R_u}\Big) a_{11} a_{12} + \Big(4\lambda_v - \tfrac{\chi^2}{R}\Big) a_{12} a_{22}$.

\paragraph{Coefficient of $v^2$ (yields $\dot{a}_{22}$).} For the terms from $H^*$ containing $v^2$, we have: $p_v(-2\beta v) \implies (2a_{22}v)(-2\beta v) \implies -4\beta a_{22}$, \quad $\lambda_m p_m^2 \implies \lambda_m (a_{12} v)^2 \implies \lambda_m a_{12}^2$, \quad $\lambda_v p_v^2 \implies \lambda_v (2a_{22} v)^2 \implies 4\lambda_v a_{22}^2$, \quad $-\tfrac{(\eta p_m + \kappa v)^2}{4R_u} \implies -\tfrac{1}{4R_u}(\eta a_{12} v + \kappa v)^2 \implies -\tfrac{1}{4R_u}(\eta^2 a_{12}^2 + 2\eta\kappa a_{12} + \kappa^2)$, and $-\tfrac{\chi^2 p_v^2}{4R} \implies -\tfrac{\chi^2}{4R} (2a_{22} v)^2 \implies -\tfrac{\chi^2}{R} a_{22}^2$.

Therefore, $\dot{a}_{22} = -4\beta a_{22} - \tfrac{\kappa^2}{4R_u} - \tfrac{\eta\kappa}{2R_u} a_{12} + \Big(\lambda_m - \tfrac{\eta^2}{4R_u}\Big) a_{12}^2 + \Big(4\lambda_v - \tfrac{\chi^2}{R}\Big) a_{22}^2$.

\paragraph{The final system.}
Collecting the six results above yields the complete Riccati ODE system in \cref{eq:Riccati}, which is solved backward from the terminal conditions $a_i(T)$ given in \cref{eq:app_terminal}.

\addtocontents{toc}{\protect\setcounter{tocdepth}{2}}
\end{appendices}

\paragraph{Acknowledgments} The author is grateful to Ivana Alexandrova and to the late Thomas B. Woolf for their helpful comments. A preliminary abstract of this paper has been accepted for presentation at the 2nd Dolomites Winter School on Mean-Field Systems in Finance, Neurosciences, and AI (January 2026), and the author thanks the organizers for the opportunity.

\bibliographystyle{plainnat}
\bibliography{main}

\end{document}